\documentclass[12pt]{amsart}
\usepackage{amssymb}
\usepackage{hyperref}
\usepackage{graphicx}

\newtheorem{theorem}{Theorem}
\newtheorem{lemma}[theorem]{Lemma}
\newtheorem{proposition}[theorem]{Proposition}

\newtheorem{corollary}[theorem]{Corollary}

\theoremstyle{definition}
\newtheorem{definition}[theorem]{Definition}
\newtheorem{notation}[theorem]{Notation}
\newtheorem{example}[theorem]{Example}

\theoremstyle{remark}
\newtheorem{remark}[theorem]{Remark}

\newcommand{\ab}{\allowbreak}
\newcommand{\ol}{\overline}
\newcommand{\ds}{\displaystyle}

\newcommand{\cA}{\mathcal{A}}
\newcommand{\bC}{\mathbb{C}}
\newcommand{\E}{\textrm{E}}
\newcommand{\gue}{{\sc gue}}
\newcommand{\N}{\mathbb{N}}
\newcommand{\cP}{\mathcal{P}}
\newcommand{\pin}{\pi_{\vec n}}
\newcommand{\cPS}{\mathcal{PS}_{NC}}
\newcommand{\tin}{\tau_{\vec n}}
\newcommand{\Tr}{\mathop{{\rm Tr}}}
\newcommand{\cV}{\mathcal{V}}
\newcommand{\vin}{\mathcal{V}_{\vec n}}
\newcommand{\cW}{\mathcal{W}}
\newcommand{\cU}{\mathcal{U}}

\title{Second Order Cumulants of Products}
\author[J. A. Mingo]{James A. Mingo$^{(*)}$}
\address{Queen's University, Department of Mathematics and
  Statistics, Jeffery Hall, Kingston, ON K7L 3N6, Canada}
\thanks{$^{*}$ Research supported by Discovery Grants and a Leadership
Support Initiative Award from the Natural Sciences and Engineering
Research Council of Canada}

\email{mingo@mast.queensu.ca}

\author[R. Speicher]{Roland Speicher$^{(*)(\dagger)}$}
\address{Queen's University, Department of Mathematics and
  Statistics, Jeffery Hall, Kingston, ON K7L 3N6, Canada}
\email{speicher@mast.queensu.ca}
\thanks{$^{\dagger}$ Research supported by a Killam Fellowship from
the Canada Council for the Arts.}

\author[E. Tan]{Edward Tan$^{(\ddagger)}$}

\address{Queen's University, Department of Mathematics and
  Statistics, Jeffery Hall, Kingston, ON K7L 3N6, Canada}

\email{3et8@qlink.queensu.ca}
\thanks{$^{\ddagger}$Research supported by a USRA from the Natural
Sciences and Engineering Research Council of Canada}
\date{}

\begin{document}

\begin{abstract}
We derive a formula which expresses a second order cumulant
whose entries are products as a sum of cumulants where the
entries are single factors. This extends to the second order
case the formula of Krawczyk and Speicher. We apply our result to the
problem of calculating the second order cumulants of a semi-circular and
Haar unitary operator. 
\end{abstract}

\maketitle

\section{Introduction}

In\footnotetext[1]{AMS classification: 46L54 (primary), 15A52,
60F05} a recent series of papers we have developed the notion of
second order freeness \cite{cmss}, \cite{kms}, \cite{mn},
\cite{mss}, \cite{ms}, \cite{emrs}. This was motivated by the
need for a framework for recent work on the global fluctuations
of the eigenvalues of ensembles of random matrices; see e.g.
Ambj{\o}rn, Jurkiewicz, and Makeenko \cite{ajm},Anderson and
Zeitouni \cite{az}, Bai and Silverstein \cite{bs}, Br\'ezin and Zee
\cite{bz}, Diaconis \cite{d}, Johansson \cite{j}, Khorunzhy,
Khoruzhenko, and Pastur \cite{kkp}.

Free independence, or what we shall call first order freeness,
was created by Voiculescu \cite{v, vdn} as an adaptation of the usual
notion of independence to the non-commutative algebra of matrix
valued (or more generally operator valued) random variables.
The central object of Voiculescu's theory is called the
$R$-transform.  For a random variable $a$, $R(z)$ is a formal
power series, which in most examples is an analytic function.
The coefficients of $R(z)$ are called the free cumulants of $a$.

The salient feature of Voiculescu's theory is that given the
moments of freely independent random variables, there is a
universal rule for calculating the moments of sums and products
of these random variables. Second order freeness achieves for the
\textit{fluctuation moments} what first order freeness does for
ordinary moments. Moreover for many of the standard ensembles of
random matrices, independent matrices are asymptotically free of
second order, so the theory is quite widely applicable.

Speicher \cite{s} developed a combinatorial approach to free
cumulants based on the non-crossing partitions of Kreweras
\cite{k}. This enabled the theory to be used in many cases
where analytic expressions could not be found. In particular in
\cite{ks}, Krawczyk and Speicher found the free analogue of the
formula of Leonov and Shiryaev \cite{ls} for calculating
cumulants where the entries are products. In this paper we
shall extend this result to the second order cumulants
introduced in \cite{cmss}. 

We begin by motivating the definition of second order cumulants.
Let $(A_N)_{N \in \N}$ be a unitarily invariant ensemble of
random matrices. By definition this means that the joint
distribution of the entries of $A_N$ is the same as the joint
distribution of the entries of $UA_NU^\ast$, where $U$ is a $N
\times N$ unitary matrix. We shall say that the ensemble
$(A_N)_{N\in \N}$ has a \textit{ second order limit
distribution} (c.f. \cite[3.4]{mss}) if for all integers $p$
and $q$ the limit

\[
\alpha_p := \lim_N \frac{1}{N} E(\Tr(A_N^p))
\]
exists, and for $Y_{N,p} = \Tr(A_N^p - \alpha_p I_N)$, the limit
\[
\alpha_{p,q} := \lim_N E(Y_{N,p} Y_{N,q})
\]
exists, and for all $r > 2$ and $p_1$, $p_2$, \dots, $p_r$
\[
\lim_N k_r( \Tr(A_N^{p_1}), \dots, \Tr(A_N^{p_r})) = 0
\]
where $k_r$ is the $r^{th}$ classical cumulant, or
semi-invariant (see e.g. \cite{ls}). Elements of the sequence
$(\alpha_p)_p$ are the moments of the limiting distribution. We
shall call the elements of the sequence $(\alpha_{p,q})_{p,q}$ the
\textit{fluctuation moments} of the limiting distribution.

These two sequences of moments may be then used to define a
{\em second order} non-commutative probability space on $\cA =
\mathbb{C}[x]$, the polynomials in the variable $x$. By a
second order non-commutative probability space we mean a triple
$(\cA, \phi, \phi_2)$ where $\cA$ is a unital algebra over
$\mathbb{C}$, $\phi : \cA \rightarrow \mathbb{C}$ is a tracial
linear functional with $\phi(1) = 1$, and $\phi_2: \cA \times
\cA \rightarrow \mathbb{C}$ is a symmetric bilinear function
which is tracial in each variable and $\phi_2 (1, a) =
\phi_2(a, 1) = 0$ for all $a \in \cA$. Thus in our example we
let $\phi(x^p) = \alpha_p$ and $\phi_2(x^p, x^q) = \alpha_{p,q}$.
 
In \cite{cmss} we introduced the second order $R$-transform,
which is a formal power series in two variables: $R(z, w) =
\sum_{p, q \geq 1} \kappa_{p,q} z^{p-1} w^{q-1}$. The coefficients
$(\kappa_{p,q})_{p,q}$ are called the second order cumulants of the
second order distribution and they depend on the first and second
order moments $(\alpha_p)_p$ and $(\alpha_{p,q})_{p,q}$ according
to the functional equation given in \cite[p. 11]{cmss}.

\begin{equation}\label{functional}
G(z, w) = G'(z) G'(w) R(G(z), G(w)) + 
\frac{\partial^2}{\partial z \partial w}
\log \Big( \frac{ G(z) - G(w) }{z - w} \Big)
\end{equation}
where
\[
G(z) = \frac{1}{z} \sum_{p \geq 0} \alpha_p z^{-p}
\mbox{ and }
G(z, w) = \frac{1}{z w} \sum_{p, q \geq 1}
\alpha_{p,q} z^{-p} w^{-q}
\]
 
Equation (\ref{functional}) determines a sequence of equations
relating the moments and the cumulants known as the moment-cumulant
relation. Below is a table giving the first few equations.

$$
\alpha_{1,1} = \kappa_{1,1} + \kappa_2
$$
$$
\alpha_{2,1} = \kappa_{1,2} + 2 \kappa_1 \kappa_{1,1} + 
2 \kappa_3 + 2 \kappa_1 \kappa_2 
$$
$$
\alpha_{2,2} = \kappa_{2,2} + 4 \kappa_1 \kappa_{2,1} 
+ 4 \kappa_1^2 \kappa_{1,1} + 4 \kappa_4 + 8 \kappa_1
\kappa_3  + 2 \kappa_2^2 + 4 \kappa_1^2 \kappa_2
$$
$$
\alpha_{1,3} 
= \kappa_{1,3} + 3 \kappa_1 \kappa_{2,1} + 3
\kappa_2 \kappa_{1,1} + 3 \kappa_1^2 \kappa_{1,1} +
3 \kappa_4 + 6 \kappa_1 \kappa_3 +
3 \kappa_2^2 + 3 \kappa_1^2 \kappa_2 
$$
$$
\alpha_{2,3} = \kappa_{2,3} +  2 \kappa_1 \kappa_{1,3} 
+ 3 \kappa_1 \kappa_{2,2} + 3 \kappa_2 \kappa_{1,2} 
+ 9 \kappa_1^2 \kappa_{1,2} + 6 \kappa_1 \kappa_2 \kappa_{1,1}
+ 6 \kappa_1^3 \kappa_{1,1}
$$
$$\mbox{}
+ 6 \kappa_5 + 18 \kappa_1 \kappa_4 
+ 12 \kappa_2 \kappa_3 + 18 \kappa_1^2 \kappa_3
+12 \kappa_1 \kappa_2^2  + 6 \kappa_1^3 \kappa_2
$$
$$
\alpha_{3,3} =
\kappa_{3,3} 
+ 6 \kappa_1 \kappa_{2,3} 
+ 6 \kappa _2 \kappa _{1,3}
+ 6 \kappa_1^2 \kappa_{1,3} 
+ 9 \kappa _1^2 \kappa_{2,2} 
+ 18 \kappa _1 \kappa _2 \kappa_{1,2}
+ 18  \kappa_1^3 \kappa_{1,2} 
$$
$$\mbox{}
+ 9 \kappa _2^2 \kappa_{1,1}
+ 18 \kappa_1^2 \kappa_2 \kappa_{1,1}
+ 9 \kappa_1^4 \kappa_{1,1}
+ 9 \kappa_6 
+ 36 \kappa_1 \kappa_5
+ 27 \kappa_2 \kappa_4
+ 54 \kappa_1^2 \kappa_4
$$
$$\mbox{}
+ 9 \kappa_3^2
+ 72 \kappa_1 \kappa _2 \kappa_3
+ 36 \kappa_1^3 \kappa_3
+ 12 \kappa_2^3
+ 36\kappa_1^2 \kappa_2^2 
+ 9 \kappa_1^4 \kappa_2    
$$

\medskip

We shall find it more convenient to use the combinatorial moment
cumulant relation for second order cumulants given in
\cite[Definition 7.4]{cmss}.

Let us recall the combinatorial definition of free
cumulants from \cite[Lecture 11]{ns}. Suppose we have a sequence of
multilinear functionals $(f_n)_n$ with $f_n : \cA \times \cdots
\times\cA \rightarrow \mathbb{C}$ being $n$-linear. We extend
this sequence to a family indexed by $\cP(n)$, the partitions
of $[n] = \{1, 2, 3, \dots, n\}$ as follows. If $V= \{ i_1,
\dots, i_k\} \subset [n]$ we let $f_V(a_1, \dots , a_n) =
f_k(a_{i_1}, \dots, a_{i_k})$, and if $\pi = \{V_1, \dots,
V_t\} \in \cP(n)$ we define 
\begin{equation}\label{product-definition}
f_\pi(a_1, \dots, a_n) = f_{V_1}(a_1, \dots, a_n) \cdots \ab
f_{V_t}(a_1, \ab \dots, a_n)
\end{equation}
 
We can now use this notation  to define the free cumulants of a
family of random variables $\{ a_1, a_2, a_3, \dots \} \subset
(\cA, \phi)$. Let $NC(n) \subset \cP(n)$ be the subset of those
partitions which are non-crossing \cite[Lecture 9]{ns}. Then we
define multilinear functionals $(\kappa_r )_r$ implicitly by
the system of equations 
\begin{equation}\label{first-moment-cumulant}
\phi(a_1 a_2 \cdots a_n) = \sum_{\pi \in NC(n)} \kappa_\pi(a_1, a_2, \dots , a_n)
\end{equation}
Note that for each $n$ this defines $\kappa_n(a_1, a_2, \dots,
a_r)$ in terms of $\kappa_p(a_{i_1}, \dots , a_{i_p})$ for $p < n$
because $\kappa_n$ only occurs once in equation
(\ref{first-moment-cumulant}), when $\pi = 1_n = \{1, 2, \dots,
n\}$.

The theorem of Krawczyk and Speicher that we extend can now be
stated. Let $(\cA, \phi)$ be a non-commutative probability space
and $a_1,\ab \dots,\ab a_{n_1}, a_{n_1 + 1}, \dots, a_{n_1 + n_2},
\dots, a_{n_1 + \cdots + n_{p-1}+1}, \dots , a_{n_1 + \cdots +
n_p}$ be elements of $\cA$. Let $A_1 = a_1 \cdots a_{n_1}$, $A_2 =
a_{n_1+1} \cdots a_{n_1 + n_2}$, \dots, $A_p = a_{n_1 + \cdots +
n_{p-1}+1} \cdots\ab a_{n_1 + \cdots +  n_p}$. The problem is to
compute $\kappa_p(A_1, A_2, \dots, A_p)$ in terms of the cumulants
$\kappa_\sigma(a_1, \dots, a_n)$ where $n = n_1 + n_2 + \cdots +
n_p$. In \cite[Theorem 2.2]{ks} it is shown that
\begin{equation}\label{ks}
\kappa_p(A_1, \dots, A_p) = \sum_{\sigma \in NC(n)} 
\kappa_\sigma(a_1, \dots, a_n)
\end{equation}
where the sum is over all non-crossing partitions $\sigma$ such
that $\sigma \vee \tau_{\vec n} = 1_n$ and $\tau_{\vec n} = \{ (1,
\dots, n_1), (n_1 + 1, \dots, n_1 + n_2), \dots, (n_1+ \cdots +
n_{p-1}+1, \dots, n_1 + \cdots + n_p)\}$ and $1_n = \{(1, \dots,
n)\}$ (see Figure 1).

\begin{figure}[t]
\begin{center}
\leavevmode\includegraphics{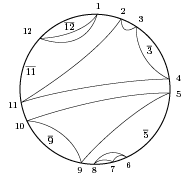}

\medskip

\leavevmode\vbox{\hsize300pt\noindent\raggedright\small{\bf
Figure 1.} In this example $n_1 =3$, $n_2 =2$, $n_3 =4$, $n_5 = 2$,
and $n_6 = 1$. $\kappa_5(a_1a_2a_3,\, a_4a_5,\, a_6a_7a_8a_9,\, a_{10}
a_{11},\, a_{12})$ will be the sum of $\kappa_\sigma(a_1, \dots,
a_{12})$'s where $\sigma$ runs over all $\sigma$ in $NC(12)$ such
that $\sigma \vee \tin = 1_{12}$ where $\tin = \{$ (1,2,3), (4,5),
(6,7,8,9), (10, 11), (12) $\}$. This condition is equivalent to the
requirement that the blocks of $\sigma$ separate the points of $\{ \ol
3,\ab \ol 5,\ab \ol 9,\ab \ol{11},\ab \ol{12}\}$. }
\end{center}
\end{figure}

The main result of this paper is to prove the analogous result for
second order cumulants, viz. to write $\kappa_{p,q}(A_1, \dots,
A_{p+q})$ in terms of first and second order cumulants of $(a_1,
\dots, a_n)$.

To describe the second order cumulants of a second order
probability space $(\cA, \phi, \phi_2)$ we need the second
order equivalent of $NC(n)$. The two parts to this extension
are (\textit{i}) the notion of a non-crossing annular
permutation (see \cite[\S3]{mn} and \cite[\S2.2]{ms}) and
(\textit{ii}) the notion of a non-crossing partitioned
permutation (see \cite[\S4]{cmss}). 
 
The non-crossing annular permutations, $S_{NC}(p,q)$, were
defined in \cite{mn} to be permutations $\pi$ in $S_{p+q}$, the
symmetric group $[p+q]$, which satisfy a geodesic condition.

\begin{equation}\label{annular-geodesic}
\#(\pi) + \#(\pi^{-1} \gamma_{p,q}) + \#(\gamma_{p,q}) = p + q
+ 2
\end{equation}
where $\#(\pi)$ denotes the number of cycles of $\pi$ and
$\gamma_{p,q}$ is the permutation with the two cycles $(1, 2, 3,
\dots, p)(p+1, \dots, p+q)$. The cycles of these permutations
can be drawn as non-crossing blocks of a $(p,q)$-annulus (see
Figure 2).

\medskip
\setbox1=\hbox{\includegraphics{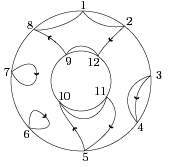}}
\setbox2=\hbox{\includegraphics{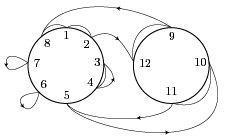}}

\begin{figure}[t]
\leavevmode\noindent\kern-3em%
$\vcenter{\hsize\wd1\box1}\vcenter{\hsize\wd2\box2}$

\begin{center}\leavevmode
\vbox{\hsize 300pt\raggedright\noindent\small
{\bf Figure 2.} A non-crossing permutation of the (8,4)-annulus
with cycles (1, 2, 12, 9, 8)(3, 4)(5, 10, 11)(6)(7). The same
permutation redrawn with the circles side-by-side.}
\end{center}
\end{figure}  

The second notion that we need is that of a non-crossing
partitioned permutation. The theory is developed in some
generality in \cite[\S4]{cmss} but we shall only need a
particular case. Let $\pi$ be a
permutation in $S_n$ and $\cV = \{ V_1, V_2, \dots, V_t\}$ be a
partition of $[n]$. If each cycle of $\pi$ is contained in some
block of $\cV$ then we say that $(\cV, \pi)$ is a
\textit{partitioned permutation}. We will frequently write this
informally as $\pi \leq \cV$.

\medskip

\begin{figure}[t]
\begin{center}\leavevmode
\includegraphics{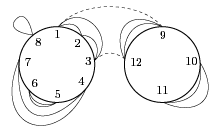}

\medskip
\leavevmode\vbox{\hsize300pt\noindent\raggedright\small
\textbf{Figure 3.} A non-crossing
partitioned permutation $(\cV, \pi)$
where $\pi = (1, 2, 3)\ab (4,7 ) \ab (5,6) (8) (9, 12) (11, 12)$
and
$\cV = \{ V_1,\ab V_2,\ab V_3,\ab  V_4, \ab V_5\}$ with $V_1
= (1, 2, 3,\ab 9, \ab 12)$, $V_2 = (4, 7)$, $V_3 = (5,\ab 6)$,
$V_4 =(8)$, and $V_5\ab = \ab (10, \ab 11)$. The block $V_1$
contains two cycles of $\pi$ --- this is indicated by the dotted
line in the diagram; all other blocks of $\cV$ contain only one
cycle of $\pi$.}

\end{center}
\end{figure}

\begin{definition}\label{non-crossing-def}
Let $p$ and $q$ be positive integers. Given $\pi \in S_{NC}(p,q)$
we let $0_\pi$ be the partition of $[p+q]$ where the blocks of
$\cV$ are exactly the cycles of $\pi$. In this way we regard $\pi$
as the partitioned permutation $(0_\pi, \pi)$. Given $\pi = \pi_1
\times \pi_2 \in NC(p) \times NC(q)$ and a partition $\cV$ of
$[p+q]$ such that $\pi \leq \cV$ we shall say $(\cV, \pi) \in
\cPS(p,q)'$ if all blocks, except one, of $\cV$ contain only one
cycle of $\pi$ and this exceptional block
contains two cycles of $\pi$ -- one of $\pi_1$ and one of $\pi_2$.
(See Figure 3.) We let $\cPS(p,q) = S_{NC}(p,q) \cup \cPS(p,q)'$
(\textit{c.f.} \cite[Definition 5.13]{cmss}).

\end{definition}

Next we wish to extend the definition in equation
(\ref{product-definition}) to the second order case.  Suppose we have
two sequences $(f_n)_n$ and
$(f_{p,q})_{p,q}$ of multi-linear functionals on a vector space
with $f_r$ being $r$-linear and invariant under cyclic permutation of its
arguments, and $f_{p,q}$ being $p+q$-linear and invariant under a cyclic
permutation of the first $p$  arguments or the last $q$ arguments.
Let $(\cV, \pi)$ be a partitioned permutation where each block of
$\cV$ contains either one or two cycles of $\pi$. If $V$ is a
block of $\cV$ containing only one cycle $(i_1,
\dots, i_s)$ of $\pi$ then we define $f_V(a_1, \dots, a_{p+q})$ to be
$f_s(a_{i_1}, \dots, a_{i_s})$, as was done for equation
(\ref{functional}). If $V$ contains two cycles $(i_1, \dots, i_s)$
and $(j_1, \dots , j_t)$ of $\pi$ then we define $f_V(a_1, \dots,
a_{p+q})$ to be $f_{s,t}(a_{i_1}, \dots, a_{i_s}, a_{j_1}, \dots,
a_{j_t})$. We then define
\begin{equation}\label{product-definition2}
f_{(\cV, \pi)}(a_1, \dots , a_{p+q}) =
f_{V_1}(a_1, \dots, a_{p+q})
 \cdots
f_{V_l}(a_1, \dots, a_{p+q})
\end{equation}
where $\cV = \{V_1, V_2, \dots, V_l \}$. 

\begin{definition}\label{second-moment-cumulant}
Let $(\cA, \phi, \phi_2)$ be a second order probability space and
$(\kappa_n)_n$ the first order cumulants given by equation
(\ref{first-moment-cumulant}). We define recursively second order
cumulants $(\kappa_{p,q})_{p,q}$ which will be multi-linear functionals
on $\cA$ by the system of equations
\begin{equation}\label{second-moment-cumulant-eq}
\phi_2(a_1 \cdots a_p, a_{p+1} \cdots a_{p+q} )
= \sum_{(\cV, \pi) \in \cPS(p,q)}
\kappa_{(\cV, \pi)}(a_1, \dots, a_{p+q})
\end{equation}
\end{definition}

In equation (\ref{second-moment-cumulant-eq}) the term
$\kappa_{p,q}(a_1, \dots, a_{p+q})$ only occurs once: when $\cV =
1_{p+q}$ and $\pi = \gamma_{p,q}$ where $\gamma_{p,q}$ is the
permutation with two cycles $(1, 2, \dots, p)(p+1, \dots, p+q)$.
For all other $(\cV, \pi)$, $\kappa_{(\cV, \pi)}$ is a product of
$\kappa_n$'s and $\kappa_{r,s}$'s with either $r < p$ or $s < q$.

It is usually convenient to write equation
(\ref{second-moment-cumulant-eq}) as a sum with two terms: the first
term only involves first order cumulants and the second term both
orders of cumulants. Recall that $\cPS(p, q) = S_{NC}(p,q) \cup
\cPS(p, q)'$ and so
\begin{eqnarray} \label{second-moment-cumulant-eq2}
\phi_2(a_1 \cdots a_p, a_{p+1} \cdots a_{p+q} )
&=&  \sum_{\pi \in S_{NC}(p,q)}
\kappa_\pi(a_1, \dots, a_{p+q}) \notag \\ 
&& \mbox{} +
\sum_{(\cV, \pi) \in \cPS(p,q)'}
\kappa_{(\cV, \pi)}(a_1, \dots, a_{p+q})
\end{eqnarray}

To illustrate this definition let us work out the first few
second order cumulants. When $p = q =1$ equation
(\ref{product-definition2}) becomes
\[
\phi_2(a, b ) = \kappa_2(a,b) + \kappa_{1,1}(a,b)
\]
because there are two elements in $\cPS(1, 1)$; the first is when
$(\cV, \pi ) = (1_2, (1,2))$ and the second when $(\cV, \pi) =
(1_2, (1)(2))$. Here we are writing the permutations
in cycle notation and $1_2$ is the partition of
$[2]$ that has one block.  By
equation (\ref{first-moment-cumulant}) $\kappa_2(a,b) = \phi(ab) -
\phi(a) \phi(b)$ . So solving for $\kappa_{1,1}(a, b)$ we obtain
that
\[
\kappa_{1,1}(a, b) = \phi_2(a,b) + \phi(a) \phi(b) - \phi(ab)
\]

Let us next look at the case $p=2$ and $q=1$. There are four
partitioned permutations $(\cV, \pi)$ with $\cV = 0_\pi$. They
are $\pi_1 = (1,3,2)$, $\pi_2 = (1,2,3)$, $\pi_3 =  (1,3)(2)$ and
$\pi_4 = (1)(2,3)$. These are the four non-crossing annular
permutations of a $(2, 1)$-annulus (see Figure 4).

\begin{figure}[t]
\begin{center}
\leavevmode\includegraphics{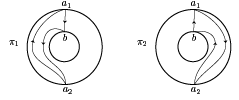}
\leavevmode\leavevmode\includegraphics{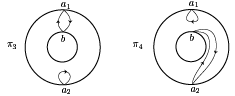}

\medskip
\leavevmode\vbox{\hsize300pt\small\noindent\raggedright
{\bf Figure 4.} The four non-crossing permutations of a $(2,
1)$-annulus.}
\end{center}\end{figure}
\medskip
The contributions of these four diagrams is $\kappa_3(a_1, b,
a_2) + \kappa_3(a_1, a_2, b) + \kappa_2(a_1, b) \kappa_1(a_2) +
\kappa_1(a_1) \kappa_2(a_2, b)$. 

There are three diagrams that involve second order cumulants;
they correspond to the three partitioned permutations:
$(1_3, \gamma_{2,1})$, $(\cV_1, e)$, and $(\cV_2, e)$, where
$1_3 = ((1,2,3))$, $\gamma_{2,1} = (1,2)(3)$, $e = (1)(2)(3)$ is
the identity permutation, $\cV_1 = \{ (1,3), (2) \}$, and $\cV_2 =
\{ (1), (2,3) \}$ (see Figure 5).  

The contribution of these three diagrams is $\kappa_{2,1}(a_1,
a_2, b) + \kappa_{1,1}(a_1, b)\ab \kappa_1(a_2) + \kappa_1(a_1)
\kappa_{1,1}(a_2, b)$. Note that second order cumulants only
appear when a block of $\cV$ connects two cycles of $\pi$.
Putting together all the terms we see that

\begin{eqnarray*}
\kappa_{2,1}(a_1, a_2, b) 
&=& \phi_2(a_1 a_2, b) - \phi(a_1)
\phi_2(a_2, b) - \phi(a_2) \phi_2(a_1, b) \\
&&\mbox{} - \phi(a_1 a_2 b) -
\phi(a_1 b a_2)  + 2 \phi(a_1) \phi(a_2 b) \\
&& \mbox{} + 2 \phi(a_1 b)
\phi(a_2) + 2\phi(a_1 a_2) \phi(b)\\
&& \mbox{} - 4 \phi(a_1) \phi(a_2)
\phi(b)
\end{eqnarray*}

In \cite[Definition 7.4]{cmss} a general formula was given for
writing any cumulant in terms of the moments using the higher
order M\"obius function. We shall prefer to use equation
(\ref{second-moment-cumulant-eq}). 

\bigskip
\begin{figure}[t]
\noindent
\kern-3em\includegraphics{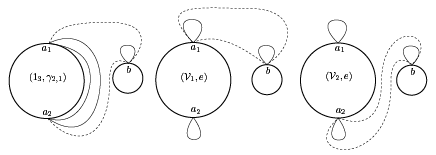}
\leavevmode\vbox{\hsize300pt\small\noindent\raggedright
{\bf Figure 5.} The three diagrams that involve second order
cumulants.}
\end{figure}
\bigskip

In equation (\ref{ks})
the condition on $\pi$ was that $\pi \vee \tau_{\vec r} = 1_n$. We
shall show that this is equivalent to the condition that $\pi^{-1}
\gamma_n$ \textit{separates} the points of $N = \{n_1, n_1 + n_2,
\dots, n_1 + \cdots + n_r
\}$, i.e. that no two points of $N$ lie in the same cycle of
$\pi^{-1}
\gamma_n$. When we pass to the annular case this condition becomes
that $\pi^{-1} \gamma_{p,q}$ separates the points of $N$.

\begin{theorem}\label{main}
Suppose $n_1, \dots , n_r, n_{r+1}, \dots, n_{r+s}$ are
positive integers,
$p = n_1 + \cdots + n_r$, $q = n_{r+1} + \cdots + n_{r+s}$, and 
\[
N = \{n_1, n_1 + n_2,\dots,  n_1 + \cdots + n_{r+s} \}
\]
Given a second order probability space
$(\cA,
\phi, \phi_2)$ and 
\[
a_1, \dots, a_{n_1}, a_{n_1 + 1}, \dots,
a_{n_1+n_2}, \dots, a_{n_1  + \cdots + n_{r+s}} \in \cA
\] 
Let
$A_1 = a_1 \cdots a_{n_1}$, \dots, $A_{r+s} = 
a_{n_1 + \cdots + n_{r+s-1} + 1}
\ab\cdots \ab a_{n_1 + \cdots + n_{r+s}}$. Then
\begin{equation}\label{main-eq}
\kappa_{r,s}(A_1, \dots, A_r, A_{r+1}, \dots, A_{r+s} )
=\sum_{(\cV, \pi)}
\kappa_{(\cV, \pi)}(a_1, \dots, a_{p+q})
\end{equation}
where the summation is over those
$(\cV, \pi) \in \cPS(p,q)$ such $\pi^{-1}\gamma_{p,q}$ separates the
points of $N$.
\end{theorem}

In section 2 we shall present some preliminary results on
non-crossing permutations needed for the proof of the main
theorem, which will appear in section 3. In section 4 we will give
some examples of how the theorem can be used and concluding remarks.

\section{preliminaries on non-crossing permutations}

In this section  we shall recall the notation we shall need for
the proof of the main theorem and translate the criterion of
equation (\ref{ks}) into a criterion appropriate to the annular
case: Proposition \ref{ksa}

\begin{notation}
Let $\pi \in S_n$ and let $\#(\pi)$ be the number of cycles in the
cycle decomposition of $\pi$. Let $\gamma_n = (1, 2, 3, \dots,
n)$. We put the metric on $S_n$ given by $|\pi| = d(e, \pi) = n -
\#(\pi)$ where $e$ is the identity permutation and $d(\pi, \sigma)=
d(e,\pi^{-1} \sigma) = d(e, \sigma \pi^{-1})$. For $\pi \in S_n$
to be in $NC(n)$ it is necessary and sufficient that
\[
\#(\pi) + \#(\pi^{-1} \gamma_n) + \#(\gamma_n) = n + 2
\]
which is equivalent to the condition in Biane \cite{b}
\[
|\pi| + |\pi^{-1}\gamma_n| = |\gamma_n|
\]
In addition, Biane showed that for $\pi$, $\sigma \in NC(n)$ the
relation $\pi \leq \sigma$ is equivalent to the equation $|\pi| +
|\pi^{-1}\sigma| = |\sigma|$, where the relation $\pi
\leq \sigma$ means that each cycle of $\pi$ is contained in some
cycle of $\sigma$. Recall that with this partial order $NC(n)$ is
a lattice; the supremum of two non-crossing permutations will be
denoted $\pi \vee \sigma$. 

For $\pi \in S_{p+q}$ to be in $S_{NC}(p,q)$ it is necessary and
sufficient (see \cite{mn}) that at least one cycle of $\pi$ connects the two
cycles of $\gamma_{p,q}$ and 
\[
\#(\pi) + \#(\pi^{-1} \gamma_{p,q}) + \#(\gamma_{p,q}) = p + q + 2 
\]
where $\gamma_{p,q}$ is the permutation with two cycles $(1, 2, 3,
\dots, p)(p+1, \dots, p+q)$. In terms of the metric this condition
becomes $|\pi| + |\pi^{-1} \gamma_{p,q}| = |\gamma_{p,q}| + 2$. In
\cite[\S 4.5]{cmss} a general theory of length of a partitioned
permutation is given. 

Let $\pi \in S_n$ and $N \subset [n]$ be a subset. If $\pi$ leaves
$N$ invariant we denote by $\pi|_N$ the restriction of $\pi$ to
$N$. If $\pi$ does not leave $N$ invariant we can still define a
permutation on $N$ as follows. If $p \in N$ and $\pi(p) \not\in
N$ then  $\pi|_N(p) = \pi^k(p)$ where $k$ is the  smallest integer
such that $\pi^{k-1}(p) \not\in N$ but $\pi^k(p) \in N$. Let $NC(N)$
denote the non-crossing permutations of $N$ where the points are in
the same order as they are in $[n]$. If $\pi$ is in $NC(n)$ then
$\pi|_N$ is in $NC(N)$, as a crossing for $\pi|_N$ would also be a
crossing for $\pi$. 

\end{notation}

Let $p$ and $q$ be positive integers and $N \subset [p + q]$.
Let $N_1 = N \cap [p]$ and $N_2 = N \cap [p+1, p+q]$. We
denote by $S_{NC}(N_1, N_2)$ the non-crossing annular permutations
of $N$ arranged on an annulus with the
points of $N_1$ on the outer circle and the points of $N_2$ on the
inner circle.  

\begin{lemma}\label{restriction}
Given $\pi \in S_{NC}(p, q)$,   $\pi|_N \in S_{NC}(N_1, N_2) \cup
(NC(p) \times NC(q))$.
\end{lemma}

\begin{proof}
By \cite[Equation 5.2]{mn} there are integers $u$ and $v$ such that
$\tilde\pi = \gamma_p^u \gamma_q^v \pi \gamma_q^{-v} \gamma_p^{-u}
\in NC(p+q)$. Let $\tilde N = \gamma_p^u \gamma_q ^v(N)$. Then
$\tilde\pi|_{\tilde N} \in NC(\tilde N)$ and $\pi|_N$ and $\tilde
\pi|_{\tilde N}$ are conjugate. Thus $\pi|_N \in S_{NC}(N_1, N_2) \cup
NC(p) \times NC(q)$.
\end{proof}

\begin{lemma}\label{invariant}
Let $N$ be a subset of $[n]$ and $\sigma, \pi \in S_n$ such that
$\pi(i) = i$ for $i \not\in N$ (and thus $\pi$ leaves $N$ invariant).
Then
$(\sigma\pi)|_N =
\sigma|_N\, \pi|_N$. 
\end{lemma}

\begin{proof}
Let $m \in N$ and suppose $(\sigma\pi)|_N(m) = (\sigma\pi)^k(m)$.
Then for $1 \leq i < k$, $\sigma^i(\pi(m)) \not\in N$  and
$\sigma^k(\pi(m)) \in N$. Thus $\sigma^i|_N (\pi(m)) \not\in N$ for
$1 \leq i < k$ and $\sigma^k(\pi(m)) \in N$. Hence $\sigma|_N(
\pi(m)) = (\sigma\pi)|_N(m)$. Hence $\sigma|_N\, \pi|_N =
(\sigma\pi)|_N$. 
\end{proof}

The next lemma is a special case of \cite[Proposition 4.10]{cmss}.
\begin{lemma}\label{transitive}
Suppose $\pi$, $\sigma$, and $\tau$ are in $S_n$ and $|\pi| +
|\pi^{-1} \sigma| = |\sigma|$ and $|\sigma| + |\sigma^{-1} \tau| =
|\tau|$. Then $|\pi| + |\pi^{-1}\tau| = |\tau|$.
\end{lemma}

In \cite[Theorem 1]{b} Biane showed that if $\sigma$ has only one
cycle and $|\pi| + |\pi^{-1} \sigma| = |\sigma|$ then the cycles of
$\pi$ form a non-crossing partition of $\sigma$. We shall need the
extension of this to the case where $\sigma$ has more than one
cycle.

\begin{lemma}\label{metric-order}
Suppose $\pi, \sigma \in S_n$ and $|\pi| + |\pi^{-1}\sigma| =
|\sigma|$. Then
\begin{enumerate}
\item each cycle of $\pi$ is contained in some cycle of $\sigma$,
and
\item for each cycle $c$ of $\sigma$ the enclosed cycles of $\pi$
form a non-crossing partition of $c$.
\end{enumerate}
\end{lemma}

\begin{proof}
First let us  show ({\it i}\/). Suppose that $a$ and $b$ are in the
same cycle of $\pi$. We must show that they are in the same cycle
of $\sigma$. If $a$ and $b$ are in the same cycle of $\pi$ then
writing the transposition that switches $a$ and $b$ as $(a, b)$ we
have $|(a, b)| + |(a, b)^{-1} \pi| = |\pi|$. Thus by Lemma
\ref{transitive} $|(a, b)| + |(a, b)^{-1} \sigma| = |\sigma|$;
hence $a$ and $b$ are in the same cycle of $\sigma$.

Let us show that ({\it ii}\/) then follows from \cite[Theorem
1]{b}.

Write $\sigma = c_1 c_2
\cdots c_k$ as a product of cycles and let $\pi_i$ be the product
of cycles of $\pi$ contained in the cycle $c_i$. For each $i$ we
have $|c_i| \leq |\pi_i| + |\pi_i^{-1} c_i|$. Thus
\begin{eqnarray*}
|\sigma| &=& |c_1| + \cdots + |c_k| \leq |\pi_1| + 
\cdots  + |\pi_k| \\
&&\mbox{} +
|\pi_1^{-1} c_1| + \cdots  + |\pi_k^{-1} c_k|  = |\pi| + |\pi^{-1}
\sigma| = |\sigma|
\end{eqnarray*}

Thus for each $i$ we must have equality in the inequality $|c_i|
\leq |\pi_i| + |\pi_i^{-1} c_i|$. Hence ({\it ii}\/).\end{proof}

\begin{definition}\label{fat} (\textit{i})
Let $n_1, n_2, \dots, n_r$ be positive integers and $n = n_1 + n_2
+ \cdots + n_r$. Given a permutation $\pi \in S_r$ we shall define
a permutation $\pin \in S_n$ as follows. Let $N = \{ n_1, n_1 +
n_2, \dots, n_1 + n_2 + \cdots + n_r \}$. For $i \not \in N$ let
$\pin(i) = i+1$ and $\pin(n_1 + \cdots + n_k) = n_1 + \cdots +
n_{\pi(k)-1} + 1$. 

\smallskip\noindent
(\textit{ii}) Let $\tau$ be the partition of $[r]$ in which all blocks
are singletons.
\end{definition}

This definition may be illustrated as follows. First $\tin = \{ T_1,
T_2, \dots,\ab T_r \}$ where $T_i$ is the cycle $(n_1 + \cdots +
n_{i-1} + 1, n_1 + \cdots + n_{i-1} + 2, \dots , n_1 + \cdots +
n_i )$. $\pin$ takes the last element of $T_i$ to the first
element of $T_{\pi(i)}$, elements not in $N$ are increased by 1. If
$\cV$ is a partition of $[r]$, we let $\vin = \cV \vee \tin$. 

\begin{example}
Let $n_1 = 2$, $n_2 = 3$, $n_3 = 4$ and $\pi = (1, 3)(2)$. Then
$\pin = (1, 2, 6, 7, 8, 9)(3, 4, 5)$ (see Figure 6.).
\end{example}

\begin{figure}[t]
\begin{center}\leavevmode\includegraphics{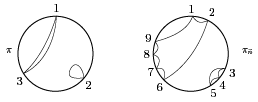}

{\small
{\bf Figure 6.} $\pi \in S_3$ and $\pin \in S_9$}
\end{center}
\end{figure}

\begin{remark}
Since $\pin$ takes the last element of  $T_i$ to the first
element of $T_{\pi(i)}$, $\pin^{-1}$ takes the first element of
$T_i$ to the last element of $T_{\pi^{-1}(i)}$. Thus
\[
\pin^{-1} \gamma_n(i) =
\begin{cases}
i & i \not\in N \\
n_1+ \cdots + n_{\pi^{-1}(j+1)} & i = n_1 + \cdots + n_j \\
\end{cases}
\]
This last case can also be written $\pin^{-1}\gamma_n (n_1 +
\cdots + n_j) = n_1 + \cdots + n_{\pi^{-1}\gamma_r(j)}$. 
\end{remark}

\begin{lemma}\label{fatnc}
Let $\psi: [r] \rightarrow [n]$ be defined by $\psi(i) = n_1 +
\cdots + n_i$. Then
\begin{enumerate}
\item $\psi \pi^{-1} \gamma_r = \pin^{-1}\gamma_n \psi$
\item $\pi \in NC(r)$ if and only if $\pin \in NC(n)$
\end{enumerate}
\end{lemma}

\begin{proof}
Part (\textit{i}) follows from the previous remark. $\#(\pi) =
\#(\pin)$ and $\#(\pin^{-1} \gamma_n) = \#(\pi^{-1}\gamma_r) +
n-r$, by part ({\it i}\/). Thus
\[
\#(\pin) + \#(\pin^{-1}\gamma_n) + \#(\gamma_n) 
=
\#(\pi) + \#(\pin^{-1}\gamma_r) + \#(\gamma_r) + n - r
\]
Hence $\#(\pin) + \#(\pin^{-1}\gamma_n) + \#(\gamma_n) = n+2$ if
and only if $\#(\pi) + \#(\pi^{-1}\gamma_r) + \#(\gamma_r) =
r + 2$
\end{proof}

\begin{lemma}\label{tracial-inequality}
Let $\tau, \sigma, \gamma \in NC(n)$ with $\tau, \sigma \leq
\gamma$. Then $\tau \leq \sigma^{-1} \gamma$ if and only if
$\sigma \leq \gamma \tau^{-1}$.
\end{lemma}

\begin{proof}
If $\tau \leq \sigma^{-1} \gamma$ then $|\tau| + |\tau^{-1}
\sigma^{-1} \gamma| = |\sigma^{-1} \gamma|$. Since $\sigma
\leq \gamma$ we also have $|\sigma| + |\sigma^{-1}\gamma| =
|\gamma|$. Thus
\begin{eqnarray*}\lefteqn{
|\sigma| + |\sigma^{-1} \gamma \tau^{-1}| 
=
|\gamma| - |\sigma^{-1} \gamma| 
+ |\tau^{-1} \sigma^{-1} \gamma| }\\ 
&=&
|\gamma| - |\sigma^{-1} \gamma| + |\sigma^{-1} \gamma| - 
|\tau| \\
&=&
|\gamma| - |\tau| = |\gamma \tau^{-1}|
\end{eqnarray*}
and therefore $\sigma \leq \gamma \tau^{-1}$.

If $\sigma \leq \gamma \tau^{-1}$ then $|\sigma| + |\sigma^{-1}
\gamma\tau^{-1}| = |\gamma \tau^{-1}|$; also $|\tau| + |\tau^{-1}
\gamma| = |\gamma|$. Thus
\begin{eqnarray*}\lefteqn{
|\tau| + |\tau^{-1} \sigma^{-1} \gamma| = |\gamma| - 
|\gamma \tau^{-1}| + |\sigma^{-1} \gamma \tau^{-1}| }\\
&=&
|\gamma| - |\gamma \tau^{-1}| + |\gamma \tau^{-1}| -|\sigma| \\
&=&
|\sigma^{-1} \gamma| = |\sigma^{-1} \gamma|
\end{eqnarray*}
which shows that $\tau \leq \sigma^{-1} \gamma$. 
\end{proof}

\begin{lemma}\label{first-sep}
For $\sigma \in NC(n)$, $\sigma \vee \tin = \gamma_n$ if and only
if $\sigma^{-1} \gamma_n$ separates the points of $N$
\end{lemma}

\begin{proof}
We have $\tin^{-1} \gamma_n (n_1 + \cdots + n_k) = n_1 + \cdots +
n_{k+1}$ and thus for $r \not= s \in N$, $(r, s) \leq \tin^{-1}
\gamma_n$ and thus by the lemma above, $\tin \leq \gamma_n(r,
s)$. Hence, if $r \not= s \in N$ are in the same cycle of
$\sigma^{-1} \gamma_n$ then by the lemma above $\sigma \leq
\gamma_n (r, s)$ and hence $\sigma \vee \tin \leq \gamma_n (r,
s)$.

Conversely if $\sigma \vee \tin < \gamma_n$ then there must be $r
\not= s \in N$ such that $\sigma \vee \tin \leq \gamma_n (r, s)$,
as partitions of the form $\gamma_n (r, s)$ with $r \not= s\in N$
are the most general partitions with two blocks, each of which is
a union of $T_i$'s. In this case $(r, s) \leq \sigma^{-1} \gamma_n$
and $\sigma^{-1}
\gamma_n$ fails to separate the points of $N$. 
\end{proof}

\begin{theorem}\label{separates}
For $\sigma \in NC(n)$ with $\sigma \leq \pin$ we have that
$\sigma \vee \tin = \pin$ if and only if $\sigma^{-1}\pin$
separates the points of $N$.
\end{theorem}

\begin{proof}
Let the cycle decomposition of $\pi$ be $c_1 \, c_2 \cdots c_k$.
Then the cycle decomposition of $\pin$ is $\tilde c_1 \, \tilde
c_2 \cdots \tilde c_k$ where $\tilde c_i = (c_i)_{\vec n}$. Let
$\sigma_i$ be the restriction of $\sigma$ to the invariant subset
consisting of the points of the cycle $\tilde c_i$; and similarly
let $\tau_i$ be the restriction of $\tin$ to the same invariant
subset. Thus $\sigma \vee \tin = \pin$ if and only if for each $i$
we have $\sigma_i \vee \tau_i = \tilde c_i$, and by the previous
lemma we have that $\sigma_i \vee \tau_i = \tilde c_i$ for each
$i$ if and only if for each $i$, $\sigma_i^{-1} \tilde c_i$
separates the points of $N \cap \tilde c_i$. Since for each $i$,
$\pin$, $\sigma$, and $\tin$ leave $\tilde c_i$ invariant, we have
that $\sigma^{-1} \pin$ separates the points of $N \cap \tilde
c_i$ for each $i$ if and only if $\sigma^{-1} \pin$ separates the
points of $N$.
\end{proof}

\begin{remark}
The condition that $\sigma^{-1} \gamma_n$ separates the points of
$N$ can be rephrased as saying that any two points in the same
cycle of $\tin^{-1}\gamma_n$ must be in different cycles of
$\sigma^{-1} \pin$. 
\end{remark}

\begin{definition}\label{order1}
Let $\pi, \sigma \in S_{NC}(p, q)$. Suppose that each cycle of
$\pi$ is contained in some cycle of $\sigma$ and for each cycle
$c$ of $\sigma$ the enclosed cycles of $\pi$ form a non-crossing
partition of $c$. Then we write $\pi \leq \sigma$. By Lemma
\ref{metric-order} this relation
extends to $S_{NC}(p, q)$ the usual partial order on $NC(n)$ given
by inclusion of blocks. 
\end{definition}

\begin{remark} 
We can rephrase the definition above using our metric $d$. Write
the cycles of $\pi$ as $c_1\, c_2\, \cdots\, c_l$ and denote the
restriction of $\pi$ to $c_i$ by $\pi_i$. Then we have
\[
\#(\pi_i) + \#(\pi_i^{-1} c_i) + \#(c_i) = k_i +2
\]
where $k_i$ is the length of the cycle $c_k$. If we sum over $i$
we get
\[
\#(\pi) + \#(\pi^{-1} \sigma) + \#(\sigma) = p + q + 2 l
\]
which can be rewritten $|\pi| + |\pi^{-1} \sigma| = |\sigma|$
 (c.f. \cite[Definition 4.9]{cmss}). That the relation $\leq$
is transitive follows from \cite[Proposition 4.10]{cmss}, or in
this simple case can be checked directly.

Moreover, we can extend the Definition \ref{order1} as follows. 
Suppose that $\pi \in S_{p+q}$, $\sigma \in S_{NC}(p,
q)$ and each cycle of $\pi$ is contained in a cycle of $\sigma$
and for each cycle of $\sigma$ the enclosed cycles of $\pi$ form a
non-crossing permutation of this cycle of $\sigma$. By
\cite[Equation 5.2]{mn} there are integers $u$ and $v$ such that
$\tilde \sigma = \gamma_p^u \gamma_q^v \sigma \gamma_q^{-v}
\gamma_p^{-u} \in NC(p+q)$. Let $\tilde\pi = \gamma_p^u \gamma_q^v
\pi \gamma_q^{-v} \gamma_p^{-u}$. Then each cycle of $\tilde\pi$
is contained in a cycle of $\tilde\sigma$ and for each cycle of
$\tilde \sigma$ the enclosed cycles of $\tilde \pi$ form a
non-crossing partition. Thus $\tilde\pi$ is non-crossing and
$\tilde\pi \leq \tilde \sigma$. Hence $\pi \in (NC(p) \times
NC(q)) \cup S_{NC}(p, q)$. This discussion  shows that the partial
order $\leq$ could also have been defined as
\begin{enumerate}
\item each cycle of $\pi$ is contained in some cycle of $\sigma$;
and 
\item whenever we have integers $u$ and $v$ such that $\gamma_p^u
\gamma_q^v \sigma \gamma_q^{-v} \gamma_p^{-u} \ab \in NC(p+q)$ we
also have $\gamma_p^u \gamma_q^v \pi \gamma_q^{-v} \gamma_p^{-u} \in
NC(p+q)$.
\end{enumerate}
\end{remark}

\begin{lemma}\label{annular-order}
Suppose $\pi, \sigma \in S_{NC}(p, q)$ with $\pi \leq \sigma^{-1}
\gamma_{p,q}$. Then $\sigma \leq \gamma_{p,q} \pi^{-1}$.
\end{lemma}

\begin{proof}
Since $\pi, \sigma \in S_{NC}(p, q)$ we have $|\pi | + |\pi^{-1}
\gamma_{p,q}| = |\gamma_{p,q}| + 2$ and $|\sigma| + |\sigma^{-1}
\gamma_{p,q}| = |\gamma_{p,q}| + 2$. By hypothesis we have $|\pi| +
|\pi^{-1} \sigma^{-1} \gamma_{p,q}| = |\sigma^{-1} \gamma_{p,q}|$. 
Thus 
\begin{eqnarray*}
|\sigma| + |\sigma^{-1} \gamma_{p,q} \pi^{-1}| 
&=&
|\gamma_{p,q}| + 2 - |\sigma^{-1} \gamma_{p,q}| +
|\sigma^{-1} \gamma_{p, q}| - |\pi| \\
&=&
|\gamma_{p,q}| + 2 - |\pi| = |\pi^{-1} \gamma_{p, q}| = |\gamma_{p,
q} \pi^{-1}|
\end{eqnarray*}
\end{proof}

In \cite[Definition 4.9]{cmss} we defined a product for partitioned 
permutations which we shall recall here. If $\cV$ is a partition of
$[n]$ then $|\cV| = n - \#(\cV)$ where $\#(\cV)$ is the number of blocks
of
$\cV$. If $(\cV, \pi)$ is a partitioned permutation then $|(\cV, \pi)| =
2 |\cV| - |\pi|$.  If $(\cV, \pi)$ and $(\cU,
\sigma)$ are partitioned permutations then they have a product if $|(\cV,
\pi)| + |(\cU, \sigma)| = |(\cV \vee \cU, \pi\sigma)|$ and we write this
as $(\cV, \pi) \cdot (\cU, \sigma) = (\cV \vee \cU, \pi\sigma)$;
otherwise the product is not defined. 

If $\sigma \in S_{p+q}$ and $\sigma$ has a cycle that connects the two
cycles of $\gamma_{p,q}$ then $\sigma \in S_{NC}(p,q)$ if and only if
$|\sigma| + |\sigma^{-1} \gamma_{p,q}| = |\gamma_{p,q}| + 2$. In terms
of partitioned permutations this means $(0_\sigma, \sigma) \cdot
(0_{\sigma^{-1} \gamma_{p,q}}, \sigma^{-1} \gamma_{p,q}) = (1_{p+q},
\gamma_{p,q})$, as requiring $1_{p+q} = 0_\sigma \vee  0_{\sigma^{-1}
\gamma_{p,q}}$ means that $\sigma$ has a cycle that connects the two
cycles of $\gamma_{p,q}$.

\begin{lemma}\label{second-annular-order}
Suppose that $\pi \in NC(p) \times NC(q)$ and $\sigma \in S_{NC}(p,
q)$ with $\pi \leq \sigma^{-1} \gamma_{p,q}$. Then 
\[(0_\sigma,
\sigma) \cdot (0_{\sigma^{-1} \gamma_{p,q} \pi^{-1}}, \sigma^{-1}
\gamma_{p,q} \pi^{-1}) = (\sigma \vee \gamma_{p,q} \pi^{-1},
\gamma_{p,q} \pi^{-1})\]
\end{lemma}

\begin{proof}
Since $\pi \leq \sigma^{-1} \gamma_{p,q}$ we have $|\pi| +
|\pi^{-1} \sigma^{-1} \gamma_{p,,q}| = |\sigma^{-1} \gamma_{p,q}|$.
Since $\pi \in NC(p) \times NC(q)$ we have $|\pi| + |\pi^{-1}
\gamma_{p,q}| = |\gamma_{p,q}|$ and since $\sigma \in S_{NC}(p,q)$
we have $|\sigma| + |\sigma^{-1} \gamma_{p,q}| = |\gamma_{p,q}| +
2$. Thus

\[
|(0_\sigma, \sigma)| + |(0_{\sigma^{-1} \gamma_{p.q}\pi^{-1}},
\sigma^{-1} \gamma_{p, q} \pi^{-1})| =
|\sigma| + |\sigma^{-1} \gamma_{p,q} \pi^{-1}| 
\]
\[ =
|\gamma_{p,q}| + 2 - |\sigma^{-1} \gamma_{p,q}| + |\sigma^{-1}
\gamma_{p,q}| - |\pi| =
|\gamma_{p,q} \pi^{-1}| + 2
\]
also
\[
|(\sigma \vee \gamma_{p,q}\pi^{-1}, \gamma_{p,q}\pi^{-1})| = 2 |
\sigma \vee \gamma_{p,q}\pi^{-1}| - |\gamma_{p,q} \pi^{-1}|
\]
so we must show that $|\sigma \vee \gamma_{p,q}\pi^{-1}| =
|\gamma_{p,q} \pi^{-1}| + 1$ i.e that $\sigma \vee
\gamma_{p,q}\pi^{-1}$ joins two cycles of $\gamma_{p,q}
\pi^{-1}$ and the cycles lie on different circles. To achieve this we
shall show that if we write $\pi = \pi_1 \times \pi_2 \in NC(p)
\times NC(q)$ then all through cycles of $\sigma$ meet only one
cycle of $\gamma_p\pi_1^{-1}$.

Indeed, suppose that two cycles $c$ and $c'$ of
$\gamma_p \pi_1^{-1}$ meet through blocks of $\sigma$. Let $a, b \in
c$  be respectively the first and last elements of $c$ (since the
cycles do not cross this is well defined). Then $a$ and $b$ are in
the same cycle of $\gamma_{p,q}\pi^{-1}$. Thus
$(\gamma_{p,q}^{-1}(a), b) \leq \pi \leq \sigma^{-1} \gamma_{p,q}$.
Hence either 
\begin{enumerate}
\item
as a partition $\sigma \leq \{(a, \gamma_{p,q}(a), \dots, b, p+1,
\dots , p+q), (\gamma_{p,q}(b), \dots,\ab \gamma_{p,q}^{-1}(a)) \}$
or
\item
as a partition $\sigma \leq \{(a, \gamma_{p,q}(a), \dots, b),
(\gamma_{p,q}(b), \dots, \gamma_{p,q}^{-1}(a), p+1,
\dots , p+q)\}$
\end{enumerate}
depending on whether the through blocks of $\sigma$ meet the cyclic
interval $(a, \dots, b)$ or the cyclic interval $(\gamma_{p,q}(b),
\dots, \gamma_{p,q}^{-1}(a))$. Since $c \subset (a, b)$ and $c'
\subset (\gamma_{p,q}(b), \dots, \gamma_{p,q}^{-1}(a))$ it is
impossible for both $c$ and $c'$ to meet thorough blocks of
$\sigma$. 
\end{proof}

\begin{corollary}\label{second-annular-order-corollary}
Given the hypotheses of Lemma \ref{second-annular-order} we have 
\begin{enumerate}
\item every cycle of $\sigma$, which is not a through cycle,  is
contained in a cycle of $\gamma_{p,q} \pi^{-1}$;
\item there are two cycles of $\gamma_{p, q} \pi^{-1}$, one from each
circle, and all the through cycles of $\sigma$ are contained in
the union of these two cycles;
\item for each cycle of $\gamma_{p, q}\pi^{-1}$ that does not meet a
through cycle of $\sigma$ the enclosed cycles of $\sigma$ form a
non-crossing permutation of this cycle of $\gamma_{p, q} \pi^{-1}$;
\item the through cycles of $\sigma$ form a non-crossing annular permutation of
the union of the two cycles of $\gamma_{p, q} \pi^{-1}$ in (\textit{ii})
above.
\end{enumerate} 
\end{corollary}

\begin{figure}[t]
\begin{center}
\leavevmode
\includegraphics{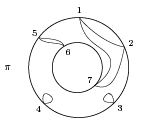} \quad
\includegraphics{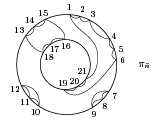}

\medskip
\leavevmode\vbox{\hsize300pt\small\noindent\raggedright
\textbf{Figure 7.} $\pi = (1,2,6)(3)(4)(5, 6) \in S_{NC}(5,2)$,
all $n_i$'s equal to 3, and $\pin = (1, 2, 3, 4, 5, 6, 19, 20, 21) 
(7, 8, 9)\ab (10, 11, 12)\ab (13,\ab 14,\ab 15, 16, 17, 18)$.}
\end{center}
\end{figure}

\begin{notation}
Let $n_1$, $n_2$, \dots, $n_{r+s}$ be positive integers and let $p
= n_1 + n_2 + \cdots + n_r$ and $q = n_{r+1} + \cdots + n_{r+s}$.
Let $N = \{ n_1, n_1 + n_2, \dots, n_1 + \cdots + n_{r+s} \}$.
Given $\pi \in S_{r+s}$ we define $\pin \in S_{p+q}$ as was done
in Definition \ref{fat}: for $i \not\in N$ $\pin(i) = i + 1$ and
$\pin(n_1 + \cdots + n_k) = n_1 + \cdots + n_{\pi(k) - 1} + 1 =
n_1 + \cdots + n_{\gamma_{r+s}^{-1}\pi(k)} + 1$. Then $\pin^{-1}
\gamma_{p,q} (n_1 + \cdots + n_k) = n_1 + \cdots + n_{\pi^{-1}
\gamma_{r,s}(k) }$.
\end{notation}

\begin{lemma}\label{intertwining}
If $\pi \in S_{NC}(r,s)$ then $\pin \in S_{NC}(p, q)$. 

\end{lemma}

\begin{proof}
Let $\psi : [r+s] \rightarrow [p+q]$ be given by $\psi(i) = n_1 +
\cdots + n_i$. Then as in Lemma \ref{fatnc} we have $\psi \pi^{-1}
\gamma_{r,s} = \pin^{-1} \gamma_{p,q} \psi$. Thus
\begin{eqnarray*}\lefteqn{
\#(\pin) + \#(\pin^{-1} \gamma_{p,q}) + \#(\gamma_{p,q}) } \\
&=&
\#(\pi) + \#(\pi^{-1} \gamma_{r,s}) + p + q - (r + s) +
\#(\gamma_{r,s}) = p + q + 2
\end{eqnarray*}

\end{proof}

\begin{proposition}\label{order}
Let $\pi \in S_{NC}(r, s)$ and $\pin \in S_{NC}(p, q)$. If $\sigma
\in S_{p+q}$ is such that 
\begin{enumerate}
\item each cycle of $\sigma$ is contained in some cycle of $\pin$;
\item for each cycle of $\pin$ the enclosed  cycles of $\sigma$
form a non-crossing permutation (see Figure 8); and
\item for each cycle $\tilde c_i$ of $\pin$ we have $\sigma_i \vee
\tau_i = \tilde c_i$, where $\sigma_i$ and $\tau_i$ are the
respective restrictions of $\sigma$ and $\tin$ to $\tilde c_i$,
\end{enumerate}
then $\sigma \in S_{NC}(p, q)$ and $\sigma^{-1} \pin$ separates
the points of $N$. 
\end{proposition}

\begin{proof}
From the remark above $\sigma \in (NC(p) \times NC(q)) \cup
S_{NC}(p, q)$. From the theorem above $\sigma^{-1} \pin$
separates the points of $N$. Since $\pi$ has a through block (i.e.
a cycle that connects the two circles) there are $k_1, l_2 \leq p$
and $k_2, l_1 > p$ all in the same cycle of $\pi$, so that
$\pi(k_1) = l_1$ and $\pi(k_2) = l_2$. Then $\pin(n_1 + \cdots +
n_{k_1}) = n_1 + \cdots n_{l_1-1} +1$ and $\pin(n_1 + \cdots +
n_{k_2}) = n_1 + \cdots n_{l_2-1} +1$. Let $\tilde c_i$ be the
cycle of $\pin$ containing $n_1 + cdots n_{k_1}$ and hence $n_1 +
\cdots n_{k_2}$. Since $\sigma^{-1} \tilde c_i$ separates the
points of $N \cap \tilde c_i$, $n_1 + \cdots + n_{k_1}$ and $n_1 +
\cdots + n_{k_2}$ must be in different cycles of $\sigma^{-1}
\tilde c_i$, hence by Lemma \ref{tracial-inequality}, $\sigma_i
\not\leq \tilde c_i  (n_1 + \cdots + n_{k_1}, n_1 + \cdots +
n_{k_2}) = d_1 d_2$ where the points on the cycle $d_1$ are
contained in $[p]$ and those of $d_2$ in $[p+1, p+q] = \{ p+1,
\dots, p+q\}$. Thus $\sigma_i$ has a through block, and so $\sigma
\in S_{NC}(p, q)$. 
\end{proof}

\begin{figure}[t]
\begin{center}
\leavevmode\includegraphics{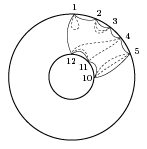}

\bigskip\leavevmode
\vbox{\hsize300pt\noindent\raggedright\small{\bf Figure 8.} 
A cycle of a non-crossing annular permutation 
$\pi$ and, in dotted lines, the enclosed cycles of a $\sigma$ with
$\sigma \leq \pi$.}
\end{center}
\end{figure}

Our last step in this section is to translate the criterion of
equation (\ref{ks}) to the annular case. We suppose we have $a_1,
\dots, a_{p+q} \in \cA$ and we let $A_1 = a_1 \cdots a_{n_1}$,
\dots $A_{r+s} = a_{n_1 + \cdots + n_{r+s-1}+1} \cdots a_{n_1 +
\cdots + n_{r+s}}$. 

\begin{proposition}\label{ksa}
Let $\pi \in S_{NC}(r, s)$ and $\pin \in S_{NC}(p, q)$. Then
\[
\kappa_\pi(A_1, \dots, A_{r+s}) =
\sum_{\sigma \in S_{NC}(p, q)}
\kappa_\sigma(a_1, \dots, a_{p+q})
\]
where the sum is over all $\sigma$ such that $\sigma^{-1}\pin$
separates the points of $N$. 
\end{proposition}

\begin{proof}
Let $c_1 \, c_2 \, \cdots \, c_k$ be the cycle decomposition of
$\pi$ and $\tilde c_1\, \tilde c_2\, \cdots \, \tilde c_k$ the
corresponding cycle decomposition of $\pin$, as in the proof of
Theorem \ref{separates}. Using the notation of equation
(\ref{product-definition})
\[
\kappa_\pi(A_1, \dots, A_{r+s}) =
\kappa_{c_1}(A_1, \dots, A_{r+s}) \cdots
\kappa_{c_k}(A_1, \dots, A_{r+s})
\]
By equation (\ref{ks}) 
\[
\kappa_{c_i}(A_1, \dots, A_{r+s}) =
\sum_{\sigma_i} \kappa_{\sigma_i}
(a_1, \dots , a_{p+q})
\]
where $\sigma_i$ runs over all non-crossing partitions of the
point set of the cycle $\tilde c_i$ such that $\sigma_i \vee
\tau_i = \tilde c_i$ (using the notation of Theorem
\ref{separates}), or equivalently that $\sigma_i^{-1} \tilde c_i$
separates the points of $N \cap \tilde c_i$. Thus
\[
\kappa_\pi(A_1, \dots, A_{r+s}) =
\sum_{\sigma_1}\cdots\sum_{\sigma_k}
\kappa_{\sigma_1}(a_1, \dots, a_{p+q}) \cdots
\kappa_{\sigma_k}(a_1, \dots, a_{p+q})
\]
where for each $i$, $\sigma_i$ runs over the non-crossing
partitions of $\tilde c_i$ such that $\sigma_i^{-1} \tilde c_i$
separates the points of $N \cap \tilde c_i$. Multiplying the
$\sigma_i$'s together we obtain all the $\sigma$'s satisfying the
hypotheses of Proposition \ref{order}. Thus the sum is over all
$\sigma \in S_{NC}(p, q)$ such that $\sigma^{-1} \pin$ separates
the points of $N$. 
\end{proof}

\section{proof of the main theorem}

Throughout this section we shall assume that we are given positive
integers $n_1, n_2, \dots , n_{r+s}$ and elements $a_1, \dots,
a_{n_1 + \cdots + n_{r+s}} \in \cA$, 
where $(\cA, \phi, \phi_2)$ is a second order
non-commutative probability space. We let $N = \{ n_1, n_1 + n_2,
\dots , n_1 + \cdots + n_{r+s} \}$, $p = n_1 + \cdots + n_r$, $q =
n_{r+1} + \cdots + n_{r+s}$, and $N_1 = [p] \cap N$ and $N_2 = N
\cap [p+1, p+q]$. We let $A_1 = a_1 \cdots a_{n_1}$, \dots, $A_{r+s}
= a_{n_1 + \cdots + n_{r + s - 1} + 1} \cdots a_{n_1 + \cdots +
n_{r+s}}$. 

We shall prove Theorem \ref{main} by induction on $(r, s)$. So to
begin let us prove the theorem when $r = s =1$.

\begin{lemma}\label{firstcase}
\begin{eqnarray*}
\kappa_{1,1}(a_1 \cdots a_p, a_{p+1} \cdots a_{p+q}) &=&
\sum_{\sigma \in S_{NC}(p,q)} \kappa_\sigma(
a_1, \dots, a_{p+q})\\
&&\quad \mbox{} +
\sum_{(\cU, \sigma)} \kappa_{(\cU, \sigma)} (a_1, \dots, a_{p+q})
\end{eqnarray*}
where in the first sum $\sigma$ is such that $p$ and $p+q$ are in
different cycles of $\sigma^{-1} \gamma_{p, q}$ and in the second
sum $(\cU, \sigma)$ runs over $\cPS(p,q)'$. 
\end{lemma}

\begin{proof}
We have two expressions for $\phi_2(a_1 \cdots a_p, a_{p+1} \cdots
a_{p+q})$. The first is
\begin{eqnarray*}
\phi_2(a_1 \cdots a_p, a_{p+1} \cdots a_{p+q}) &=&
\sum_{\sigma \in S_{NC}(p, q)}
\kappa_\sigma(a_1, \dots , a_{p+q})\\ 
&& \mbox{} +
\sum_{(\cU, \sigma) \in \cPS(p, q)'}
\kappa_{(\cU, \sigma)}(a_1, \dots , a_{p+q}) \\
\end{eqnarray*}
and the second is
\begin{eqnarray*}
\phi_2(a_1 \cdots a_p, a_{p+1} \cdots a_{p+q}) &=&
\kappa_2(a_1 \cdots a_p, a_{p+1} \cdots a_{p+q})\\ 
&& \mbox{} +
\kappa_{1,1}(a_1 \cdots a_p, a_{p+1} \cdots a_{p+q}) \\
\end{eqnarray*}
Combining these two with equation (\ref{ks}) we have
\begin{eqnarray}\lefteqn{
\kappa_{1,1}(a_1 \cdots a_p, a_{p+1} \cdots a_{p+q})
= \sum_{\sigma \in S_{NC}(p, q)}
\kappa_\sigma(a_1, \dots , a_{p+q})}\\
&& \mbox{} - 
\sum_{\sigma  \in NC(p+q)} \kappa_\sigma(a_1, \dots
a_{p+q}) + 
\sum_{(\cU, \sigma)}
\kappa_{(\cU, \sigma)}(a_1, \dots , a_{p+q})\notag
\end{eqnarray}
where the second sum, by  Lemma \ref{first-sep}, is over $\sigma \in
NC(p+q)$ such that $p$ and $p+q$ are in different cycles of
$\sigma^{-1} \gamma_{p+q}$.

Now each $\sigma \in NC(p+q)$ such that $p$ and $p+q$ are in
different cycles of $\sigma^{-1} \gamma_{p+q}$ must have a through
block by Lemma \ref{tracial-inequality} (i.e.\@ a block connecting a
point in $[p]$ to one in $[p+1, p+q]$) and thus by
\cite[Equation 5.2]{mn} is a also a non-crossing $(p, q)$-annular
permutation. Finally since $\sigma^{-1} \gamma_{p,q} = \sigma^{-1}
\gamma_{p+q} (p, p+q)$, we have that $p$ and $p+q$ are in the same
cycle of $\sigma^{-1}\gamma_{p,q}$. Hence the $\sigma$'s in the
second sum are exactly those in $S_{NC}(p, q)$ for which $p$ and
$p+q$ are in the same cycle of $\sigma^{-1} \gamma_{p, q}$.
Removing these from the first sum we are left with the $\sigma$'s
for which $p$ and $p+q$ are in different cycles of $\sigma^{-1}
\gamma_{p,q}$. 

Note that for the term 
\[
\sum_{(\cU, \sigma)} \kappa_{(\cU, \sigma)}(a_1, \dots , a_{p+q})
\]
$\sigma = \sigma_1 \times \sigma_2 \in NC(p) \times NC(q)$ so that
the condition that $p$ and $p+q$ are in different cycles of
$\sigma^{-1} \gamma_{p,q} = \sigma_1^{-1} \gamma_p \times \sigma_2
\gamma_q$ is automatically satisfied.

\end{proof}

In order to give an inductive proof of our main theorem we shall
need an extension of its statement to the case of partitioned
permutations.

\begin{lemma}\label{inductive}
Suppose that Equation $($\ref{main-eq}\,$)$ holds for $r' \leq r$
and
$s'
\leq s$. Then for $(\cV, \pi) \in \cPS(r,s)'$ we have
\[
\kappa_{(\cV, \pi)} (A_1 , \dots , A_{r+s}) = \kern-0.75em
\sum_{ \sigma \in S_{NC}(p,q)} \kern-0.75em
\kappa_\sigma (a_1, \dots, a_{p+q}) +
\sum_{(\cU, \sigma)} \kappa_{(\cU, \sigma)}
(a_1, \dots , a_{p+q})
\]
where the first sum is over all $\sigma$'s such that $\sigma \leq
\vin$ (i.e. each cycle of $\sigma$ is contained in a block of
$\vin$) and $\sigma^{-1} \pin$ separates the points of $N$,
and the second sum is over all $(\cU, \sigma)$ such that
$\cU \leq \vin$ and
$\sigma^{-1} \pin$ separates the points of $N$.
\end{lemma}

\begin{proof}
Let us write $\pi$ as a product of cycles: $c_1 \cdots c_k c' c''$
where the cycles $c_1$, \dots, $c_k$ are the cycles of $\pi$ that
are the only cycle in the block of $\cV$ that contains them and $c'$
and $c''$ are the two cycles of $\pi$ that lie in a single block of
$\cV$. Let $1_{c' \cup c''}$ be this block of $\cV$
that contains the two cycles $c'$ and $c''$. Then by
Equation (\ref{main-eq})
\begin{eqnarray*}\lefteqn{
\kappa_{(\cV, \pi)}(A_1, \dots, A_{r+s}) =
\kappa_{c_1}(A_1, \dots, A_{r+a}) \cdots 
\kappa_{c_k}(A_1, \dots, A_{r+s}) }\\
&&\mbox{} \times 
\kappa_{(1_{c' \cup c''}, c' c'')}(A_1, \dots , A_{r+s}) \\
&=&
\sum_{\sigma_1, \dots, \sigma_k} 
\kappa_{\sigma_1}(a_1, \dots, a_{p+q}) \cdots
\kappa_{\sigma_k}(a_1, \dots, a_{p+q}) \\
&&\mbox{} \times \Big\{
\sum_{\sigma_0} \kappa_{\sigma_0}(a_1, \dots, a_{p+q}) +
\sum_{(\cU_0, \sigma_{00})}
\kappa_{(\cU_0, \sigma_{00})}(a_1, \dots, a_{p+q}) \Big\}
\end{eqnarray*}
where $\sigma_i$ runs over $NC(\tilde c_i)$ such that $\sigma^{-1}
\tilde c_i$ separates the points of $N \cap \tilde c_i$; $\sigma_0$
runs over $S_{NC}(\tilde c', \tilde c'')$ such that $\sigma_0^{-1}
\tilde c' \tilde c''$ separates the points of $N \cap 1_{\tilde c'
\cup \tilde c''}$; and $(\cU_0, \sigma_{00})$ runs
over $\cPS(\tilde c',\tilde c'')$ such that $\sigma_{00}^{-1}
\tilde c' \tilde c''$ separates the points of $N \cap 1_{\tilde c'
\cup \tilde c''}$. The product $\sigma_0 \sigma_1 \cdots \sigma_k$
thus runs over $\sigma \in S_{NC}(p, q)$ such that
$\sigma^{-1}\pin$ separates the points of $N$ and $\sigma \leq
\vin$. The product $(\cU_0, \sigma_{00}) \sigma_1
\cdots \sigma_k$ runs over $(\cU, \sigma) \in \cPS(p,q)'$ such that
$\sigma^{-1} \pin$ separates the points of $N$ and $\cU \leq \vin$.
This proves our assertion. 
\end{proof}

In the next lemma we consider $\sigma \in S_{NC}(p, q)$ such that
$\sigma^{-1} \gamma_{p, q}$ \textit{fails} to separate the points
of $N$, we wish to determine $\pi \in S_{NC}(r, s)$ such that
$\pin$ will behave like the supremum of $\sigma$ and $\tin$ in
that with respect to $\pin$, $\sigma$ will satisfy the three
conditions of Proposition \ref{order}.

\begin{lemma}\label{first-order-uniqueness}
Let $\sigma \in S_{NC}(p, q)$ be such that
$\sigma^{-1}\gamma_{p,q}$ has a cycle that meets both $N_1 = N
\cap [n]$ and $N_2 = N \cap [p+1, p+q]$. Then there is a unique
$\pi \in S_{NC}(r, s)$ such that relative to $\pi$, $\sigma$
satisfies the three conditions of Proposition \ref{order}.
\end{lemma}

\begin{proof}
Let us first deal with uniqueness. Suppose $\pi \in S_{NC}(r,s)$ is
such that $\sigma^{-1} \pin$ separates the points of $N$. Then
$(\sigma^{-1} \pin)|_N = \textit{id}_N$ and for $i \not\in N$,
$\pin^{-1} \gamma_{p, q} (i) = i$. Thus by Lemma \ref{invariant},
$(\sigma^{-1} \gamma_{p,q})|_N = (\sigma^{-1} \pin)|_N\cdot
(\pin^{-1} \gamma_{p,q})|N = (\pin^{-1} \gamma_{p, q})|_N$. Since
$\pin^{-1} \gamma_{p,q} \psi = \psi \pi^{-1} \gamma_{r,s}$ by Lemma
\ref{intertwining}, $\pi$ is uniquely determined by
$(\sigma^{-1}
\gamma_{p, q})|_N$. 

To show existence consider $(\sigma^{-1} \gamma_{p, q})|_N$.
By Lemma \ref{invariant}, $(\sigma^{-1} \gamma_{p, q})|_N
\in S_{NC}(N_1, N_2) \cup (NC(N_1) \times NC(N_2))$.
By assumption, $\sigma^{-1} \gamma_{p,q}$ has a cycle meeting
both $N_1$ and $N_2$; thus $(\sigma^{-1} \gamma_{p,q})|_N \in
S_{NC}(N_1, N_2)$. Then there is $\pi \in S_{NC}(r, s)$ such
that $\psi \pi^{-1} \gamma_{r, s}\psi^{-1}  = (\sigma^{-1}
\gamma_{p, q})|_N$.

Now $\pin^{-1} \gamma_{p,q}(i) = i$ for $i \not \in N$ and $\pin^{-1}
\gamma_{p,q}|_N = \sigma^{-1} \gamma_{p,q} |_N$ together imply that
$\pin^{-1} \gamma_{p,q} \leq \sigma^{-1} \gamma_{p,q}$ and thus by
Lemma \ref{annular-order}, $\sigma \leq \pin$ and the cycles of
$\sigma$ form a non-crossing partition of those of $\pin$. Finally
by applying Lemma \ref{first-sep} to each cycle of $\pin$ we obtain
that the three conditions of Proposition \ref{order} are satisfied. See
Figure 9.
\end{proof}


\begin{figure}[t]
\begin{center}
\leavevmode
\hbox{\includegraphics{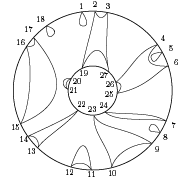}\hfill\includegraphics{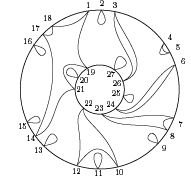}}

\

\hbox to \hsize{
\includegraphics{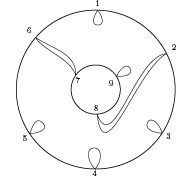} \hfill\includegraphics{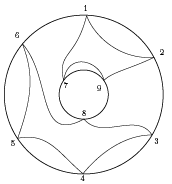}}

\bigskip\leavevmode
\vbox{\hsize300pt\noindent\small
\raggedright
{\bf Figure 9.} 
Shown is $\sigma \in S_{NC}(18,9)$ in the upper left,
$\sigma^{-1}\gamma_{18,9}$ in the upper right. In the lower left we have
$\pi^{-1} \gamma_{6,3}$ and the lower right we have $\pi \in S_{NC}(6,3)$.
$\pin$ is shown overleaf. Note that when $N = \{ 3, 6, 9, 12, 15,
18, 21, 24, 27 \}$, $\sigma$ satisfies the
hypotheses of Lemma \ref{first-order-uniqueness}, for example 6
and 24 are in the same cycle of $\sigma^{-1} \gamma_{18,9}|_N$.
The resulting $\pi^{-1} \gamma_{6,3}$ has 2 and 8 in the same
cycle and thus $\pi \in S_{NC}(6,3)$. }
\end{center}
\end{figure}

\begin{lemma} \label{second-order-uniqueness}
Suppose that $\sigma \in S_{NC}(p, q)$ is such
that $\sigma^{-1}  \gamma_{p,q}$ does not separate the
points of $N$ but no cycle of $\sigma^{-1}\gamma_{p, q}$
meets both $N_1$ and $N_2$. Then there is a unique $(\cV,
\pi) \in \cPS(r,s)$ such that $\vin = \sigma \vee \tin$ and
$\sigma^{-1} \pin $ separates the points of $N$.
\end{lemma}

\begin{proof}
Let us first deal with uniqueness. $\vin$ is determined by the
equation $\vin = \sigma \vee \tin$. If $\sigma^{-1} \pi$ separates
the points of $N$ then $\sigma^{-1} \pin|_N = \textit{id}_N$. Also
$\pin^{-1} \gamma_{p,q}$ is the identity on $N^c$, thus by Lemma
\ref{invariant} $\sigma^{-1} \gamma_{p,q} |_N = \sigma^{-1} \pin|_N
\cdot \pin^{-1} \gamma_{p,q}|_N = \pin^{-1} \gamma_{p,q}|_N$. And
since $\pin^{-1} \gamma_{p,q} \psi = \psi \pi^{-1} \gamma_{r,s}$ we
see that $\pi$ is uniquely determined. 

Secondly let us prove existence. Let us denote by $NC(N_1)$ and
$NC(N_2)$ respectively the non-crossing partitions of $N_1$ and $N_2$
where the order is that determined by $\gamma_{p, q}$. Similarly
$S_{NC}(N_1, N_2)$ denotes the non-crossing annular permutations on
$N_1 \cup N_2$, again with the order determined by $\gamma_{p, q}$. 

By Lemma \ref{restriction}, 
$$\sigma^{-1} \gamma_{p,q}|_N \in S_{NC}(N_1, N_2) \cup ( NC(N_1)
\times NC(N_2))$$ By assumption no cycle $\sigma^{-1}
\gamma_{p,q}|_N$ meets both $N_1$ and $N_2$, so $\sigma^{-1}
\gamma_{p,q}|_N \in NC(N_1) \times NC(N_2)$. Thus there is $\pi \in
NC(r) \times NC(s)$ such that $\psi \pi^{-1} \gamma_{r,s} \psi^{-1}
= \sigma^{-1} \gamma_{p,q} |_N$, and thus $\sigma^{-1} \gamma_{p,q}
|_N = \pin^{-1} \gamma_{p,q} |_N$. By Lemma \ref{invariant},
$\sigma^{-1} \pin|_N = \textit{id}_N$, so $\sigma^{-1} \pin$
separates the points of $N$.


\begin{figure}[t]
\begin{center}
\leavevmode
\hbox{\includegraphics{fig8a.png}\hfill\includegraphics{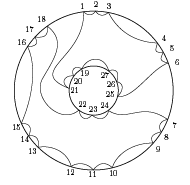}}

\bigskip\leavevmode
\vbox{\hsize300pt\noindent\small
\raggedright
{\bf Figure 10.} 
$\sigma$ (left)  from Figure 9 is compared with the $\pin \in
S_{NC}(18,9)$ produced by Lemma \ref{first-order-uniqueness}.}
\end{center}
\end{figure}

Recall that $\pin^{-1}\gamma_{p,q}(i) = i$ for $i \not \in N$ and
$\pi^{-1} \gamma_{p,q} |_N = \sigma^{-1} \gamma_{p,q} |_N$ thus
$\pin^{-1} \gamma_{p, q} \leq \sigma^{-1} \gamma_{p,q}$, and so by
Corollary \ref{second-annular-order-corollary}, 
$\sigma \vee \pin$ is obtained from the cycles of $\pin$ by joining the
two cycles of $\pin$, one from each circle, that meet the through cycles
of $\sigma$. To complete the proof we must show that $\sigma
\vee \tin = \sigma \vee \pin$, for this we must show that $\pin \leq
\sigma \vee \tin$ i.e. that any two blocks of $\tin$ connected by
$\pin$ are connected by $\sigma$. 

Let $c$ be a cycle of $\pin$ not meeting a through cycle of
$\sigma$ and let $\tilde \sigma$ the product of the cycles of
$\sigma$ enclosed by $c$. Then $\tilde\sigma^{-1} c$ separates the
points of $N$ in $c$ and so by Lemma \ref{first-sep} any two blocks
of $\tin$ in $c$ connected by $c$ are connected by $\tilde \sigma$.

Now let $c$ and $c'$ be the two cycles (one in each cycle of
$\gamma_{p,q}$) that meet the through cycles of $\sigma$ and let
$\tilde\sigma$ be the product of the cycles of $\sigma$ enclosed by
$c \cup c'$ and the same for $\tilde \tau$. Now we must show that
$\tilde \sigma \vee \tilde \tau = c \cup c'$. If not, then there
would be $r \not=s \in N$ such that $(r, s) \leq \tilde \sigma^{-1}
c c'$ in which case $r$ and $s$ are in the same cycle of
$\sigma^{-1}\pin$ contrary to our hypothesis. Hence $\sigma \vee
\tin = \pin$.
\end{proof}

\medskip\noindent\textit{%
Proof of Theorem 3:} Let us write $\phi_2(A_1 \cdots A_r, A_{r+1}
\cdots A_{r+s})$ two different ways:

\begin{eqnarray*}\lefteqn{
\phi_2(A_1 \cdots A_r, A_{r+1} \cdots A_{r+s}) 
=
\sum_{\pi \in S_{NC}(r,s)} \kappa_\pi(A_1, \dots, A_{r+s}) }\\
&& \mbox +
\sum_{(\cV, \pi) \in \cPS(p,q)'}
\kappa_{(\cV, \pi)}(A_1, \dots, A_{r+s})\\
&=&
\sum_{\pi \in S_{NC}(r,s)} \kappa_\pi(A_1, \dots, A_{r+s}) \\
&& \mbox +
\sum_{(\cV, \pi) \in \cPS(p,q)''}
\kappa_{(\cV, \pi)}(A_1, \dots, A_{r+s}) \\
&& \mbox{} +
\kappa_{r,s}(A_1, \dots, A_{r+s})
\end{eqnarray*}
where $\cPS(r, s)'' = \cPS(r,s)' \setminus \{(1_{r+s}, \gamma_{r,
s}) \}$. Also
\begin{eqnarray*}\lefteqn{
\phi_2(A_1 \cdots A_r, A_{r+1} \cdots A_{r+s})
= \phi_2(a_1 \cdots a_p, a_{p+1} \cdots a_{p+q}) }\\
&=&
\sum_{\pi \in S_{NC}(p, q)}
\kappa_\pi(a_1, \dots, a_{p+q}) \\
&& \mbox{} +
\sum_{(\cV, \pi) \in \cPS(p, q)'}
\kappa_{(\cV, \pi)}(a_1, \dots, a_{p+q}) \\
\end{eqnarray*}


\begin{figure}[t]
\begin{center}
\leavevmode
\hbox to \hsize
{\includegraphics{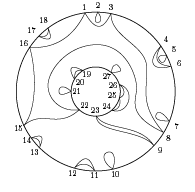}\hfill\includegraphics{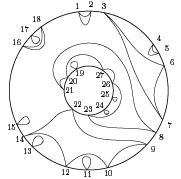}}

\

\hbox to \hsize%
{\includegraphics{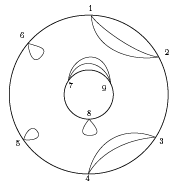}\hfill\includegraphics{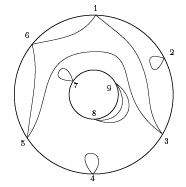}}

\bigskip\leavevmode
\vbox{\hsize350pt\noindent\small
\raggedright
{\bf Figure 11.} 
Shown is  $\sigma \in S_{NC}(18,9)$ in the upper left,
$\sigma^{-1}\gamma_{18,9}$ in the upper right. In the lower left we have
$\pi^{-1} \gamma_{6,3}$ and the lower right we have $\pi \in S_{NC}(6,3)$.
$\pin$ is shown overleaf. Note that when $N = \{ 3, 6, 9, 12, 15, 18,
21, 24, 27 \}$, $\sigma$ satisfies the
hypotheses of Lemma \ref{second-order-uniqueness}, for example 3
and 6 are in the same cycle of $\sigma^{-1} \gamma_{18,9}|_N$.
The resulting $\pi^{-1} \gamma_{6,3}$ has 1 and 2 in the same
cycle. }
\end{center}
\end{figure}

Now solving for $\kappa_{r,s}(A_1, \dots, A_{r+s})$ we have
\begin{eqnarray*}
\kappa_{r, s}(A_1, \dots, A_{r+s}) 
&=&
\sum_{\sigma \in S_{NC}(p,q)}
\kappa_\sigma(a_1, \dots, a_{p+q}) \\
&& \mbox{}-
\sum_{\pi \in S_{NC}(r, s)}
\kappa_\pi(A_1, \dots, A_{r+s}) \\
&& \mbox{} +
\sum_{(\cV, \pi) \in \cPS(p, q)}
\kappa_{(\cV, \pi)}(a_1, \dots, a_{p+q}) \\
&&\mbox{} -
\sum_{(\cV, \pi) \in \cPS(r, s)''} 
\kappa_{(\cV, \pi)}(A_1, \dots, A_{r+s})
\end{eqnarray*}

Now by Proposition \ref{ksa}
\begin{eqnarray*}
\sum_{\pi \in S_{NC}(r, s)}
\kappa_\pi(A_1, \dots, A_{r+s})
&=&
\sum_{\pi \in S_{NC}(r, s)}
\mathop{\sum_{\sigma \in S_{NC}(p, q)}}_%
{\sigma^{-1}\pin\ {\rm sep.\hbox{\tiny'} s}\ N}
\kappa_\sigma(a_1, \dots, a_{p+q})
\end{eqnarray*}
If $\pi \in S_{NC}(r, s)$ and $\sigma^{-1} \pin$ separates the
points of $N$ then $\sigma^{-1} \gamma_{p,q}|_N =
\pin^{-1}\gamma_{p,q}|_N$ by Lemma \ref{invariant} and
$\pin^{-1}\gamma_{p,q}$ has a cycle that meets both $N_1$ and $N_2$.
Conversely by Lemma \ref{first-order-uniqueness} if $\sigma \in
S_{NC}(p, q)$ is such that $\sigma^{-1} \gamma_{p,q}|_N$ has a
cycle that meets both $N_1$ and $N_2$ then $\sigma$ comes from a
unique $\pi \in S_{NC}(r, s)$. Thus


\begin{figure}[t]
\begin{center}
\leavevmode
\hbox{\includegraphics{fig9a.png}\hfill\includegraphics{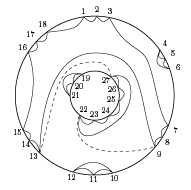}}

\bigskip\leavevmode
\vbox{\hsize350pt\noindent\small
\raggedright
{\bf Figure 12.} 
$\sigma$ (left)  from Figure 11 is compared with the $\pin \in
NC(18)\times NC(9)$ produced by Lemma \ref{second-order-uniqueness}. The
two cycles of $\pin$ that meet through cycles of $\sigma$ are shown
connected by dotted lines.}
\end{center}
\end{figure}

\begin{eqnarray*}
\sum_{\sigma \in S_{NC}(p, q)} 
\kappa_\sigma(a_1, \dots, a_{p+q})
&-&
\sum_{\pi \in S_{NC}(r, s)}
\mathop{\sum_{\sigma \in S_{NC}(p, q)}}_%
{\sigma^{-1}\pin\ {\rm sep.\hbox{\tiny'} s}\ N}
\kappa_\sigma(a_1, \dots, a_{p+q}) \\
&&=
\mathop{\sum_{\sigma \in S_{NC}(p, q)}}_%
{\sigma^{-1}\gamma_{p,q}\ {\rm sep.\hbox{\tiny'} s}\ N}
\kappa_\sigma(a_1, \dots, a_{p+q})
\end{eqnarray*}

To conclude the proof we use induction on $r$ and $s$. In Lemma
\ref{firstcase} we have proved equation (\ref{main-eq}) when $r = s =
1$. We can apply the induction hypothesis to $\cPS(r, s)''$ because
$(1_{r+s}, \gamma_{r,s})$ has been removed. Thus by Lemma 
\ref{inductive} we have:

\begin{eqnarray*}\lefteqn{
\sum_{(\cV, \pi) \in \cPS(r,s)''}
\kappa_{(\cV, \pi)} (A_1, \dots, A_{r+s})} \\
&=&
\sum_{(\cV, \pi) \in \cPS(r,s)''}
\mathop{\sum_{(\cU, \sigma) \in \cPS(p,q)'}}_%
{\sigma^{-1}\pin\ {\rm sep.\hbox{\tiny'} s}\ N}
\kappa_{(\cU, \sigma)}(a_1, \dots, a_{p+q}) \\
&=&
\mathop{\sum_{(\cU, \sigma) \in \cPS(p,q)'}}_%
{\sigma^{-1}\gamma_{p,q}\ {\rm does\ not\ sep.}\ N}
\kappa_{(\cU, \sigma)}(a_1, \dots, a_{p+q})
\end{eqnarray*}
where in the last sum not separating $N$ means that no cycle of
$\sigma^{-1}\gamma_{p,q}$ meets both $N_1$ and $N_2$ but some cycle
of $\sigma^{-1}\gamma_{p,q}$ contains more than one point of $N$.

Hence 
$$
\sum_{(\cV, \pi) \in \cPS(p, q)'}
\kappa_{(\cV, \pi)}(a_1, \dots, a_{p+q})
\kern0.5em - \kern-1em
\sum_{(\cV, \pi) \in \cPS(r,s)''}
\kappa_{(\cV, \pi)}(A_1, \dots, A_{r+s})
$$
$$
=
\mathop{\sum_{(\cU, \sigma) \in \cPS(p, q)'}}_%
{\sigma^{-1} \gamma_{p,q} {\rm sep.\hbox{\tiny'}s}\ N}
\kappa_{(\cU, \sigma)}(a_1, \dots ,a_{p+q})
$$
This proves Equation (\ref{main-eq}). \qed

\section{First example -- the square of a semi-circular operator}

Let $(\cA, \phi)$ be a unital *-algebra with $\phi : \cA
\rightarrow \bC$ a state. A self-adjoint element $x \in \cA$ is
called \textit{semi-circular} if $\phi(x^{2k-1}) = 0$ and
$\phi(x^{2k}) = \frac{1}{k+1} \binom{2k}{k}$ for $k = 1, 2, 3,
\dots$ The element $x$ is called semi-circular because the density
of its spectral measure with respect to $\phi$ is $\frac{1}{2 \pi}
\sqrt{4 - x^2}$. This density was shown by E. Wigner to be the
limiting eigenvalue distribution of various ensembles of random
matrices, in particular on ensemble usually known as the Gaussian
Unitary ensemble (\gue). This ensemble can be described as follows.

For each positive integer $N$ let $X_N = (f_{i,j})$ be the complex
$N \times N$ self-adjoint random matrix such that
\begin{itemize}
\item $f_{i,j} = x_{ij} + \sqrt{-1} y_{ij}$;

\item for $i \not = j$, $x_{ij}$ and $y_{ij}$ are
real Gaussian random variables with mean 0 and variance $1/(2N)$;

\item for each $i$, $x_{ii}$ is a real Gaussian random
variable with mean 0 and variance $1/N$

\item $\{ x_{ij} \}_{i \leq j} \cup \{ y_{ij} \}_{i <
j}$ is an independent set of random variables
\end{itemize}
Other descriptions exist, see for example Deift \cite[\S 5.2]{de} or
Hiai and Petz \cite[\S 4.1]{hp}.

The theorem of Wigner asserts that the limiting
eigenvalue distribution of $X_N$ has the semi-circular density
given above. In particular the limiting moments $(\alpha_k)$ are
those of a semi-circular operator: $\alpha_{2k-1} = 0$ and
$\alpha_{2k} = \frac{1}{k+1} \binom{2 k}{k}$ for $k \geq 1$. 

In \cite{j} Johansson considered the asymptotic behaviour of the
random variables $\{ \Tr( X_N^k - \alpha_k I_N) \}_k$ for
a class of ensembles containing the \gue. He showed that the
random variables were asymptotically Gaussian and showed that the
Chebyshev polynomials of the first kind diagonalized the
covariance.

In \cite{mn} Mingo and Nica showed that the limiting covariances
$\alpha_{p, q} = \lim_N \E( \Tr( X_N^p - \alpha_p I_N) \Tr( X_N^q -
\alpha_q I_N) )$, sometimes called the \textit{fluctuation
moments}, were positive integers which count the number of
non-crossing pairings of a $(p, q)$-annulus. These are the elements
of $S_{NC}(p, q)$ for which all cycles are of length 2. These
diagrams or equivalent formulations had already been used in a
variety of earlier papers: Tutte \cite{t} and \cite{t2}, Jones
\cite{jo}, and King \cite{ki} where it was found that

\begin{eqnarray}\label{semi-circular-fluctuations}\lefteqn{
\alpha_{p,q} = \sum_{k \geq 1} k \binom{ p}{\frac{p-k}{2}}
\binom{q}{\frac{q-k}{2}} } \\
&=& \begin{cases}
\ds \frac{p q}{2p + 2 q} 
\binom{p}{\frac{p}{2}} \binom{q}{\frac{q}{2}}
& \vrule width 0pt depth 2.5em 
p \mbox{\ and\ } q \mbox{\ are even} \\
\ds \frac{(p +1)(q +1)}{8p + 8 q} 
\binom{p+1}{\frac{p+1}{2}}
\binom{q+1}{\frac{q+1}{2}}
& p \mbox{\ and\ } q \mbox{\ are odd}
\end{cases} \notag
\end{eqnarray}

The first expression for $\alpha_{p,q}$ can be seen by observing
that every non-crossing pairing of a $(p, q)$-annulus with $k$
through strings (i.e. strings that connect the two circles) can be
obtained by connecting a non-crossing pairing of $[p]$ with one
block of size $k$ and the others of size 2, with a non-crossing
partition of $[q]$ with one block of size $k$ and the others of size
2; and then invoking Kreweras \cite[Th\'eor\`eme 2]{k} and
finally summing over $k$.

It is natural then to define $(\alpha_{p,q})_{p,q}$ to be the
fluctuation moments of the semi-circular operator.

Let $\cA = \bC[x]$ be the polynomials in $x$. Define $\phi: \cA
\rightarrow \bC$ by $\phi(x^k) = \alpha_k$ where $(\alpha_k)_k$ are
the moments of the semi-circular operator. Define $\phi_2(x^p, x^q)
= 0$ if $p =0$, $q=0$, or $p+q$ is odd and $\phi_2(x^p,
x^q) =\alpha_{p,q}$ (as defined in Equation
(\ref{semi-circular-fluctuations})) if $p+q$ is even. One then
extends $\phi$ and $\phi_2$ by linearity. Notice that the choice of
$\phi_2$ is independent of $\phi$. 

Now $(\cA, \phi, \phi_2)$ is a second order non-commutative
probability space and we can consider the cumulants of $x$. It is a
standard calculation to see that $\kappa_n(x, x, \dots, x)$ is 0
for all $n \not = 2$ and $\kappa_2(x, x) = 1$. In the case
of second order cumulants, $(\kappa_{p,q})_{p,q}$ are all 0. We
can see this directly from Equation
(\ref{second-moment-cumulant-eq2}) which states:
\begin{eqnarray}\label{second-moment-cumulant-eq3}\lefteqn{
\phi_2(x^p, x^q) =
\sum_{\pi \in S_{NC}(p,q)} \kappa_\pi(x, \dots, x) }\\
&&\mbox{} +
\sum_{(\cV, \pi) \in \cPS(p,q)'}
\kappa_{(\cV, \pi)}(x, \dots, x) \notag
\end{eqnarray}

Because $x$ is semi-circular, $\kappa_\pi(x, \dots, x)$ is 0 unless
$\pi$ is a pairing in which case $\kappa_\pi(x, \dots, x)$ is 1.
Thus the first term on the right hand side of
(\ref{second-moment-cumulant-eq3}) is $\alpha_{p,q}$. Since
by definition the left hand side is also $\alpha_{p,q}$ we have
that $\sum_{(\cV, \pi) \in \cPS(p,q)'} \kappa_{(\cV, \pi)} (x,
\dots, x) =0$ for all $p$ and $q$. Hence for $p = q = 1$,
$\kappa_{1,1}(x, \dots, x) =0$ as there is only one term. Now
$\kappa_{p,q}(x, \dots, x)$ only appears once in the expression
$\sum_{(\cV, \pi) \in \cPS(p,q)'} \kappa_{(\cV, \pi)} (x,
\dots, x)$, all the other terms have a factor of $\kappa_{r,s}(x,
\dots, x)$ for either $r < p$ and $s \leq q$ or $r \leq p$ and $s <
q$. Thus by induction $\kappa_{p,q}(x, \dots, x) = 0$ for all $p$
and $q$. 

Now let $a = x^2$. We shall use Theorem \ref{main} to show that the
second order cumulants $\kappa_{p,q} = \kappa_{p,q}(a, \dots, a)$
have the generating function
\[
C(z, w) = \sum_{p,q \geq 1} \kappa_{p,q} z^p w^q = 
\frac{z w}{(1 - z - w)^2}
\] and thus 
\[
\kappa_{p,q} = p \binom{p+q -1}{p}
\]

First let us recall the first order cumulants of $a$. Let $N = \{2,
4, 6, \dots, 2n \}$.  By Equation
(\ref{ks}) and Theorem \ref{separates} 
\[
\kappa_n(a, \dots, a) = \kappa_n(x^2, \dots, x^2) 
=\kern-1em
\mathop{\sum_{\pi \in NC(2n)}}_%
{\pi^{-1}\gamma_{2n}\ {\rm sep.\hbox{\tiny'} s}\ N}
\kappa_\pi(x, \dots, x)
\]


\begin{figure}[t]
\begin{center}
\leavevmode\includegraphics{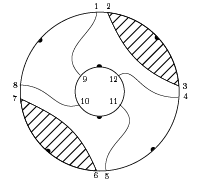}

\bigskip\leavevmode
\vbox{\hsize300pt\noindent\small
\raggedright
{\bf Figure 13.} 
An example of a pairing $\pi$ of the $(8, 4)$-annulus such that
$\pi^{-1} \gamma_{8,4}$ separates the points of $\{ 2, 4, 6, 8, 10 \}$.
To construct such a pairing we choose $p -k$ dots on one circle and $q
-k$ dots on the other circle and cover these with a pairing (dots 2
and 6 covered with the shading above). All remaining points are
connected with through strings. In this example there are only two ways
to make the connections: 1 is connected to either 9 or 11; this choice
forces all the others.}
\end{center}
\end{figure}

Now as noted above $\kappa_\pi(x, \dots, x) = 0$ unless $\pi$ is a
pairing. Let $\pi$ be a pairing such that $\pi^{-1} \gamma_{2n}$
separates the points of $N$. Since $\pi$ is non-crossing 2 is connected to
$2 k + 1$ for some $k$. Then 2 and $2 k$ are
in the same cycle of $\pi^{-1} \gamma_{2n}$ unless $k = 1$, i.e. 2
is connected to 3. In general we must have that $2 k$ is connected
to $2 k +1$ for $1  \leq k < n$ and $2n$ is connected to 1. Thus
there is only one pairing such that $\pi^{-1} \gamma_{2n}$
separates the points of $N$. The contribution of this pairing is 1.
Hence $\kappa_n(a, \dots, a) = 1$ for all $n$.

Now let us find the second order cumulants of $a$. Fix $p$ and $q$
and let $N = \{2, 4, 6, \dots, 2p + 2 q\}$. 
\begin{eqnarray*}\lefteqn{
\kappa_{p,q}(a, \dots, a) = \kappa_{p,q}(x^2, \dots, x^2) }\\
&=&
\mathop{\sum_{\pi \in S_{NC}(2p, 2q)}}_%
{\pi^{-1}\gamma_{2p,2q}\ {\rm sep.\hbox{\tiny'} s}\ N}
\kappa_\pi(x, \dots, x) 
+\kern-1em
\mathop{\sum_{(\cV, \pi) \in \cPS(2p, 2q)'}}_%
{\pi^{-1}\gamma_{2p,2q}\ {\rm sep.\hbox{\tiny'} s}\ N}
\kappa_{(\cV, \pi)}(x, \dots, x) \\
&=&
\mathop{\sum_{\pi \in S_{NC}(2p, 2q)}}_%
{\pi^{-1}\gamma_{2p,2q}\ {\rm sep.\hbox{\tiny'} s}\ N}
\kappa_\pi(x, \dots, x) 
\end{eqnarray*}
because, as we have shown, all the second order cumulants of $x$
are 0. Moreover as noted above $\kappa_\pi(x, \dots, x) = 0$ except
for $\pi$ a pairing in which case we get 1. Thus $\kappa_{p,q}(a,
\dots, a)$ is the cardinality of $\{ \pi \in S_{NC}(2p, 2q) \mid
\pi$ is a pairing and $\pi^{-1} \gamma_{2p, 2q}$ separates the
points of $N\}$. 

\begin{proposition}
$\ds\kappa_{p,q}(a, \dots, a) =
\sum_{k \geq 1} k \binom{p}{k} \binom{q}{k}$

\end{proposition}

\begin{proof}
We shall show that the number of non-crossing pairings $\pi$ of a $(2p,
2q)$-annulus with $k$ pairs of through strings such that $\pi^{-1}\gamma_{2p,
2q}$ separates the points of $N = \{2, 4, 6, \dots, 2p + 2q\}$ is $k
\binom{p}{k}\binom{q}{k}$. Summing over $k$ establishes our claim. 

Let $\pi$ be a  non-crossing pairing $\pi$ of a $(2p, 2q)$-annulus  with $k$
pairs of through strings such that $\pi^{-1}\gamma_{2p, 2q}$ separates the
points of $N$.  By the same argument as in the disc case the only possible
pairs of $\pi$ which are not through strings are of the form $(2t, 2 t+1)$.
There are $\binom{p}{k} = \binom{p}{p-k}$ ways to place these $p-k$ pairs of
adjacent elements on $[2p]$ and $\binom{q}{k} = \binom{q}{q-k}$ ways of
placing these $q-k$ pairs of adjacent elements on $[2p+1, 2p+2q]$. In
principle there are $2k$ ways of connecting the $2k$ through strings but
half of these put two elements of $N$ in the same cycle of $\pi^{-1}
\gamma_{2p, 2q}$, so in practice there are only $k$ ways of connecting the
through strings. Thus the number of pairings is indeed $k \binom{p}{k}
\binom{q}{k}$. See Figure XX. 

\end{proof}

\begin{theorem}
\[
\sum_{p,q \geq 1} \kappa_{p,q}(a, \dots, a) z^p w^q
= \frac{z w }{(1 - z - w)^2}
=\sum_{p, q \geq 1} p \binom{p+q-1}{p} z^p w^q
\]
\end{theorem}

\begin{proof}
\begin{eqnarray*}\lefteqn{
\sum_{p,q \geq 1}\kappa_{p,q}(a, \dots, a) z^p w^q
=
\sum_{p, q \geq 1} \sum_{k \geq 1}
k \binom{p}{k} \binom{q}{k} z^p w^q }\\
&=&
\sum_{k \geq 1} k \sum_{p\geq 1} \binom{p}{k} z^p
\sum_{q \geq 1} \binom{q}{k} w^q
=
\sum_{k \geq 1}k \frac{z^k}{(1 - z)^{k+1}} \frac{w^k}{(1 - w)^{k+1}}\\
&=&
\frac{z w}{[(1 -z)(1-w)]^2} \sum_{k \geq 1} k \bigg( \frac{zw}{(1-z)(1-w)}
\bigg)^{k-1} \\
&=&
\frac{z w}{[(1 -z)(1-w)]^2}
\frac{1}{( 1 - \frac{zw}{(1-z)(1-w)})^2} 
= \frac{z w}{(1 - z - w)^2}
\end{eqnarray*}
Continuing we have
\begin{eqnarray*}
\frac{z w}{(1 - z - w)^2} &=& \sum_{k \geq 1} k (z + w)^{k-1} z w 
=
\sum_{k \geq 1} \sum_{p=0}^{k-1} k \binom{k-1}{p} z^{p+1} w^{k -p} \\
&=&
\sum_{k, p \geq 1} k \binom{k-1}{p-1} z^p w^{k+1 -p}\\
&=&
\sum_{p, q \geq 1} (p+q-1) \binom{p+q-2}{p-1} z^p w^q 
\ \left\{\vcenter{\hsize 60pt\noindent
\textit{letting}
$q = k+1-p$}\right. \\
&=&
\sum_{p,q \geq 1} p \binom{p+q-1}{p} z^p w^q
\end{eqnarray*}
\end{proof}

Let us next outline a method for establishing the same result that uses
the second order $R$-transform in Equation (\ref{functional}) or the
equivalent formulation given in \cite[Equation (52)]{cmss}. Let $C(z, w)
= 1 + z w R(z, w)$, and $M(z) = \sum_{k\geq 0} \phi(a^k)z ^k$ be the
moment generating function of $a$, and $M(z, w) = \sum_{p,q \geq 1}
\phi_2(a^p, a^q) z^p w^q$ be the fluctuation moment generating function
of $a$. Then Equation (\ref{functional}) becomes 

\begin{eqnarray}\label{functional2}
C(z M(z), w M(w) ) 
&=& \bigg( M( z, w) + \frac{z w}{(z - w)^2} \bigg)
\frac{M(z)}{ \frac{d}{dz}(z M(z)) }
\frac{M(w)}{ \frac{d}{dw}(w M(w)) }\notag\\
&&\mbox{} - 
\frac{ z M(z) w M(w)}{(z M(z) - w M(w))^2}
\end{eqnarray}

We must find the fluctuation moment
generating function $F(z, w)$ of  the semi-circular
operator. Let 
\[
F(z) = \sum_{k\geq0} \phi(x^k) = \frac{1 - \sqrt{1 - 4 z^2}}{2 z^2}
\eqno \mbox{and}
\]
\[
F(z, w) = \sum_{p, q \geq 1} \phi_2(x^p, x^q) z^p w^q =
\frac{z \frac{d}{dz}(z F(z)) w \frac{d}{dz}(w F(w))}%
{(1 - z F(z) w F(w))^2}
\]
\[
= \frac{ z w (z F(z) - wF(w))^2}{(z-w)^2 
(1 - z^2 F(z)^2)(1 - w^2 F(w)^2)}
\]
Where we obtain the second last expression for $F(z, w)$ by using the first
expression for $\alpha_{p,q}$ in
Equation (\ref{semi-circular-fluctuations}), 
then Lambert's identity \cite[equation 5.21]{gkp}, and finally summing
over $k$, and the last expression is obtained from the quadratic
equation satisfied by $F$.  

Now
\[ M(z^2) = F(z) \ \mbox{and} \ 
M(z^2, w^2) = \frac{1}{2}( F(z, w) + F(-z, w) )
\]
Thus
\[
M(z, w) =
\frac{2 z w M(z)^3 M(w)^3}{(2 - M(z))(2 - M(w)) 
(1 - z M(z)^2 w M(w)^2 )^2 }
\]
If we make the substitution $u = z M(z)$ and $v = w M(w)$ then we have
\[
M(z) = \frac{1}{1 -u},\ z = u - u^2,\  
\frac{d}{dz}(z M(z)) = \frac{1}{1 - 2 u},\ \mbox{and}\
\]
\[
M(z, w) = \frac{2 u v (1 - u)(1 - v)}{(1 - 2 u)(1 - 2 v)(1 - u -v)^2}
\]
Thus
\[ M(z, w) + \frac{z w}{(z - w)^2} =
\frac{u v (1 - u)(1 - v)(1 - 2 u - 2 v + 2 u^2 + 2 v^2)}%
{(1 - 2 u)(1 - 2 v)(u - v)^2(1 - u - v)^2}
\]
After some routine manipulations it follows from
(\ref{functional2}) and  the equation above that
\[
C(u, v) = \frac{u v}{ (1 - u - v)^2}
\]

\section{Second Example -- a Haar Unitary}

Let $(\cA, \phi)$ be a non-commutative probability space and $u \in \cA$
a unitary. Recall (see e.g. \cite[Lecture 10]{ns}) that $u$ is a Haar
unitary if $\phi(u^k) = 0$ for $k \not = 0$. In \cite{ns} it is shown
that for $\epsilon_1, \epsilon_2, \dots , \epsilon_n \in  \{-1, 1\}$ then
$\kappa_n(u^{\epsilon_1}, u^{\epsilon_2}, \dots, u^{\epsilon_n}) = 0$
unless $n$ is even and $\epsilon_1 + \epsilon_2 = \epsilon_2 +
\epsilon_3 = \cdots = \epsilon_{n - 1} + \epsilon_n = 0$; i.e. all
the free cumulants of $\{u, u^\ast \}$ are 0 except for the
alternating ones: $\kappa_n(u, u^\ast, \dots, u,
u^\ast)=\kappa_n(u^\ast, u,
\dots, u^\ast, u)$ which equal
$\mu(0_n, 1_n) =(-1)^{n-1} c_{n-1}$ where $\mu$ is the M\"obius function
of the lattice $NC(n)$ and $c_n$ is the $n^{th}$ Catalan number. We wish
here to indicate the corresponding result for the second order
cumulants of $u$ and $u^\ast$. The proofs will be given in a another
paper.

We first have to decide how to define the fluctuation moments of
a Haar unitary. In \cite[Theorem 2]{ds}, Diaconis and Shahshahani
showed that if $U_N$ is an $N \times N$ Haar distributed random
unitary then  \[\E( \Tr(U_N^k) \Tr(U_N^{-k})) = |k| \mbox{\ for\ }
N \geq 2\] 
We shall use these fluctuation moments to define our
second order Haar unitary. 

\begin{definition}
Let $(\cA, \phi, \phi_2)$ be a second order probability space 
and $u \in \cA$ a unitary. We say that $u$ is a second order Haar
unitary if $\phi(u^k) = 0$ for $k \not = 0$ (i.e. $u$ is a Haar
unitary) and for all integers $k$, $\phi_2(u^k, u^{l}) =
\delta_{k,-l} |k|$. 
\end{definition}

\begin{proposition}
Let $p$ and $q$ be positive integers and $u$ a second order
Haar unitary. Let $\epsilon_1, \epsilon_2, \epsilon_3, \dots ,
\epsilon_{p+q} \in \{-1, 1\}$. Then $\kappa_{p,q}(u^{\epsilon_1},
u^{\epsilon_2}, \dots,\ab u^{\epsilon_{n-1}}, u^{\epsilon_n}) =0$
unless $p$ and $q$ are even and 
\begin{equation}\label{p-part}
\epsilon_1 + \epsilon_2 = \cdots =
\epsilon_{p-1} + \epsilon_p =
\epsilon_{p+1} + \epsilon_{p+2} = \cdots =
\epsilon_{p+q-1} + \epsilon_{p+q} = 0
\end{equation}
i.e. the $\epsilon$'s alternate in sign except possibly between $p$
and $p+1$. 
\end{proposition}

As noted above, in the first order case the alternating cumulants of $u$
and $u^\ast$ were given by the M\"obius function $\mu$ of the lattice
$NC(n)$, i.e.
\[
\kappa_{2n}(u, u^\ast, \dots, u, u^\ast) = 
\kappa_{2n}(u^\ast, u, \dots, u^\ast, u) = \mu(0_n, 1_n)
\]

We shall state an analogous result for the second order case.
Recall from \cite[\S 5.4]{cmss} that the second order M\"obius
function is defined as the multiplicative function which is the
convolution inverse of the zeta function on the set of partitioned
permutations. It was shown \cite[Theorem 5.24]{cmss}
that $\mu(1_n, \gamma_n) = (-1)^{n-1} c_{n-1}$ and
$\mu(1_{p+q}, \gamma_{p,q}) = (-1)^{p+q} c_{p,q}$ where $c_{p,q}$ is the
cardinality of $S_{NC}(p, q)$. Moreover it was shown  \cite[page
46]{cmss} that the the M\"obius function satisfies the
following recurrence relation
\begin{eqnarray*}
0 &=&
\mu(1_{p+q}, \gamma_{p,q}) + q \mu(1_{p+q}, \gamma_{p+q}) \notag \\
&& \mbox{} + 
\sum_{1 \leq k < p}\Big(
\mu(1_{k+q}, \gamma_{k,q}) \mu(1_{p-k}, \gamma_{p-k}) +
\mu(1_k, \gamma_k) \mu(1_{p-k+q}, \gamma_{p-k,q}) \Big)
\end{eqnarray*}

The alternating cumulants of $u$ and $u^\ast$
satisfy the same recurrence relation.

\begin{theorem}
Let $p = 2m$ and $q = 2n$ be even integers and $\epsilon_1, \dots,
\epsilon_{p+q}
\in \{-1, 1\}$ satisfy (\ref{p-part}). Then
$\kappa_{p,q}(u^{\epsilon_1}, \dots, u^{\epsilon_{p+q}}) =
(-1)^{m+n} c_{m, n}$.

\end{theorem}

\section{concluding remarks -- a partial order on $\cPS(p,q)$}

In \S 2 we proved two lemmas (\ref{annular-order}
and \ref{second-annular-order}) that show that if
$\pi$ is a sub-partition of the Kreweras complement
$\sigma^{-1}\gamma_{p,q}$ of $\sigma$, then $\sigma$ is a sub-partition of the
Kreweras complement $\gamma_{p,q} \pi^{-1}$ of $\pi$, provided we take
the complement on the other side. In Lemma \ref{second-annular-order} 
this order was expressed terms of multiplication of partitioned permutations.
The multiplication of partitioned permutations can be used to make $\cPS(p,q)$
into a partially ordered set. This order is an extension of the one
given in Definition \ref{order1}.

\begin{definition}
If $(\cV, \pi), (\cU, \sigma) \in \cPS(p,q)$ we say that $(\cV, \pi) \leq
(\cU, \sigma)$ if there is $\cW$ such that $(\cV, \pi) \cdot (\cW,
\pi^{-1} \sigma) = (\cU, \sigma)$. 
\end{definition}

\begin{proposition}\label{order-structure}
Let $(\cV, \pi), (\cU, \sigma) \in \cPS(p,q)$ and suppose there is $\cW$ such
that $(\cV, \pi) \cdot (\cW, \pi^{-1} \sigma) = (\cU, \sigma)$. Then 
\begin{enumerate}
\item $\cW = 0_{\pi^{-1} \sigma}$ and
\item $\cU = \cV \vee \pi^{-1} \sigma = \cV \vee \pi \vee \sigma = \cV \vee
\sigma\pi^{-1}$;
\item $(0_{\sigma\pi^{-1}}, \sigma \pi^{-1}) \cdot (\cV, \pi) = (\cU, \sigma)$.
\end{enumerate}
\end{proposition}

\begin{proof}
Recall that for $(\cV, \pi) \in \cPS(p,q)$ either $\pi \in S_{NC}(p, q)$ and
$\cV = 0_\pi$ or $\pi \in NC(p) \times NC(q)$ and all the blocks of $\cV$
contain just one cycle of $\pi$ except one block of $\cV$ which contains two
cycles of $\pi$ -- one from each circle; thus in this case $|\cV| = |\pi| + 1$.
Hence in the former case $|(\cV, \pi)| = |\pi|$ and in the latter case
$|(\cV, \pi) = |\pi| + 2$. 

Thus $|(\cV, \pi)| = |\pi| + \delta$ where $\delta = 0$ or 2. By assumption
$|(\cV, \pi)| + |(\cW, \pi^{-1} \sigma)| = |(\cU, \sigma)|$. If $\sigma \in
S_{NC}(p, q)$ then $\cU = 0_\sigma$ and $|(\cU, \sigma)| = |\sigma|$. Thus
\[
|\pi| + \delta + |(\cW, \pi^{-1} \sigma)| = |\sigma| \leq |\pi| + |\pi^{-1}
\sigma|
\]
Therefore $\delta = 0$ and $|(\cW, \pi^{-1} \sigma)| = |\pi^{-1}
\sigma|$ and so $\cW = 0_{\pi^{-1} \sigma}$.

Suppose that $\sigma \in NC(p) \times NC(q)$ then $|\cU| = |\sigma|
+ 1$ and $|(\cU, \sigma)| = |\sigma| + 2$. If $\pi \in NC(p) \times
NC(q)$ then $|(\cV, \pi)| = |\pi| + 2$ and again we can use the
triangle inequality 
\[
|\pi| + 2 + |(\cW, \pi^{-1} \sigma)| = |\sigma| + 2 \leq |\pi| + 2 +
|\pi^{-1} \sigma|
\]
to conclude that $\cW = 0_{\pi^{-1} \sigma}$. When $\pi \in
S_{NC}(p,q)$ we have to appeal to \cite[Proposition 5.11]{cmss} and
for this we need a little preparation.

Let the cycles of $\sigma$ be $c_1$, $c_2$, \dots , $c_{k+1}$ and
the blocks of $\cU$ be $\{ U_1, \dots, U_k\}$ with $U_i = c_i$ as
sets for $1 \leq i < k$ and $U_k = c_k \cup c_{k+1}$ also as sets.
Then we can write $\pi = \pi_1 \, \pi_2 \cdots \pi_k$ when $\pi$ is
the product of the cycles of $\pi$ contained in $U_i$. We can also
decompose $\cV$ and $\cW$ into $\cV_1$, $\cV_2$, \dots , $\cV_k$ and
$\cW$ into $\cW_1$, $\cW_2$, \dots, $\cW_k$ along the blocks of
$\cU$. By the triangle inequality we have that 

\begin{equation}\label{triangle-inequalities}
\begin{cases}
|(U_i, c_i)| \leq |(\cV_i, \pi_i)| + |(\cW_i, \pi^{-1}_i c_i)|
 \mbox{\ \ for\ } 1 \leq i < k \mbox{\ and} & \\
|(U_k, c_k c_{k+1})| \leq |(\cV_k, \pi_k)| +
|(\cW_k, \pi_k^{-1} c_k c_{k+1})| & \\
\end{cases}
\end{equation}

Directly from the definition of $| \cdot |$ we obtain that
\[
|(\cV, \pi)| = |(\cV_1, \pi_1)| + \cdots + |(\cV_k, \pi_k)|
\]
\[
|(\cW, \sigma)| = |(\cW_1, \pi_1^{-1} c_1)| + 
\cdots + |(\cW_k, \pi_k^{-1} c_k c_{k+1})|
\]
and
\[
|(\cU, \sigma)| = |(\cU_1, c_1)| + \cdots + |(\cU_k, c_k c_{k+1})|
\]
Hence
\begin{eqnarray*}
|(\cV_1, \pi_1)| &+& \cdots + |(\cV_k, \pi_k)| +
|(\cW_1, \pi_1^{-1} c_1)| + 
\cdots \\
&& \mbox{} + |(\cW_k, \pi_k^{-1} c_k c_{k+1})| =
|(\cU_1, c_1)| + \cdots + |(\cU_k, c_k c_{k+1})|
\end{eqnarray*}
If there were a strict inequality in any of the inequalities
(\ref{triangle-inequalities}) then we would have strict inequality
in the equality above. Since we don't, all the inequalities
in (\ref{triangle-inequalities}) must be equalities. Hence
\[
|(U_i, c_i)| = |(\cV_i, \pi_i)| + |(\cW_i, \pi^{-1}_i c_i)|
 \mbox{\ \ for\ } 1 \leq i < k \mbox{\ and} 
\]
\[
|(U_k, c_k c_{k+1})| = |(\cV_k, \pi_k)| +
|(\cW_k, \pi_k^{-1} c_k c_{k+1})| 
\]
Thus for each $i < k$, $\cW_i = 0_{\pi^{-1}_i c_i}$ and $\pi_i$ is
a non-crossing partition of $c_i$. The last equality shows that
$(\cV_k, \pi_k) \cdot (\cW_k, \pi^{-1}_k c_k c_{k+1}) = (\cU_k, c_k
c_{k+1})$. By \cite[Proposition 5.11]{cmss} we have that $\cW_k =
0_{\pi^{-1}_k c_k c_{k+1}}$ and thus $\cW = 0_{\pi^{-1}
\sigma}$. Moreover $(\cV_k, \pi_k)$ is in $\cPS(c_k, c_{k+1})$.
This proves (\textit{i}).

Since $\pi$ and $\sigma$ are permutations the orbits of the
subgroups generated by $\{ \pi, \sigma \}$, $\{\pi, \pi^{-1} \sigma
\}$, and $\{ \pi, \sigma\pi^{-1} \}$ are all the same. Thus as
partitions $\pi \vee \sigma = \pi \vee \pi^{-1} \sigma = \pi \vee
\sigma \pi^{-1}$. Since $\pi \leq \cV$ we have $\cU = \cV \vee
\pi^{-1} \sigma = \cV \vee \pi \vee \sigma = \cV \vee \pi \vee
\sigma \pi^{-1} = \cV \vee \sigma \pi^{-1}$. This proves
(\textit{ii}) and (\textit{iii}) follows from (\textit{ii}).
\end{proof}

\begin{remark}
Associativity of multiplication of partitioned permutations
(\cite[Proposition 4.10]{cmss}) shows that the partial order is
transitive. Indeed, suppose $(\cV, \pi), (\cU, \sigma), (\cW, \tau)
\in \cPS(p,q)$ and that $(\cV, \pi) \leq (\cU, \sigma)$
and $(\cU, \sigma) \leq (\cW, \tau)$. Then
\[
(\cV, \pi) \cdot (0_{\pi^{-1} \sigma}, \pi^{-1} \sigma)
= (\cU, \sigma)
\mbox{\ and\ }
(\cU, \sigma) \cdot (0_{\sigma^{-1} \tau}, \sigma^{-1} \tau) =
(\cW, \tau)
\]
Then
\[
(\cV, \pi) \cdot (0_{\pi^{-1} \sigma} \vee 0_{\sigma^{-1} \tau},
\pi^{-1} \tau) = (\cW, \tau)
\]
and incidentally from Proposition \ref{order-structure} that as
partitions $\pi^{-1} \sigma\vee \sigma^{-1} \tau = \pi^{-1}
\tau$
\end{remark}

\begin{corollary}
Let $(\cV, \pi), (\cU, \sigma) \in \cPS(p,q)$. Then
$(\cV,  \pi)\leq (\cU, \sigma)$ if and only if
\begin{enumerate}
\item $\cV \leq \cU$;

\item if $\sigma \in S_{NC}(p,q)$ then $\pi \in S_{NC}(p,q)$ and
every cycle of $\pi$ is contained in a cycle of $\sigma$ and for each
cycle of $\sigma$ the enclosed cycles of $\pi$ form a non-crossing
permutation of this cycle of $\sigma$;

\item if $\sigma \in NC(p) \times NC(q)$ then every cycle of $\pi$
is contained in either a cycle of $\sigma$ or the union of the two
cycles of $\sigma$ connected by $\cU$; and for every cycle of
$\sigma$ or the union of the two cycles joined by $\cU$ the
enclosed cycles of $\pi$ form a non-crossing permutation of this
cycle or union of two cycles.
\end{enumerate}
\end{corollary}

\begin{remark}
In \cite[Notation 5.11]{cmss} a partitioned permutation $(\cV,
\pi)$ with $\cV = 0_\pi$ was called a \textit{disc permutation} and one
where $|\cV| = |\pi| + 1$ a \textit{tunnel permutation}. With our order
on $\cPS(p,q)$ we see that we can have: (\textit{i}) $\mbox{disc} \leq
\mbox{disc}$, (\textit{ii}) $\mbox{disc} \leq \mbox{tunnel}$; and
(\textit{iii}) $\mbox{tunnel} \leq \mbox{tunnel}$; but 
$\mbox{tunnel} \not\leq \mbox{disc}$.
\end{remark}

\begin{remark}
With this order, $\cPS(p,q)$ becomes a partially ordered set and
the M\"obius function of the poset $\cPS(p,q)$ has a simple relation to
the M\"obius function introduced in \cite[\S 5.4 ]{cmss} and used in \S
5. We will address this relation in a forthcoming paper.
\end{remark}


\bibliographystyle{plain}

\end{document}